\documentclass[a4paper]{article}

\usepackage{amsthm,amsmath,amssymb,pxfonts}

\usepackage{mathrsfs}
\usepackage[margin=1in]{geometry}
\usepackage{here}
\usepackage{authblk}
\usepackage[usenames]{color}

\usepackage[dvipdfmx,%
bookmarks=true,%
bookmarksnumbered=true,%
colorlinks=true,%
linkcolor=NavyBlue,%
setpagesize=false,%
pdftitle={Limit distributions for the discretization error of stochastic Volterra equations},%
pdfauthor={Masaaki Fukasawa, Takuto Ugai},%
pdfsubject={paper},%
pdfkeywords={Stochastic Volterra Equation}]{hyperref}

\newcommand{\Natu}{\mathbb{N}}

\newcommand{\Real}{\mathbb{R}}

\newcommand{\mr}[1]{\mathrm{#1}}

\newcommand{\EXP}[1]{\mathsf{E}\left[#1\right]}
\newcommand{\PROB}[1]{\mathsf{P}\left[#1\right]}
\newcommand{\mcl}[1]{\mathcal{#1}}
\newcommand{\filt}[1]{\mathscr{#1}}
\newcommand{\msc}[1]{\mathscr{#1}}
\newcommand{\hd}{H\" older }
\newcommand{\hds}{H\" older's }
\newcommand{\U}{\mathbf{U}}
\newcommand{\V}{\mathbf{V}}

\title{Limit distributions for the discretization error of \\
stochastic Volterra equations}
\author{Masaaki Fukasawa}
\author{Takuto Ugai}
\affil{Graduate School of Engineering Science, Osaka University}
\date{\vspace{-5ex}}

\numberwithin{equation}{section}

\theoremstyle{plain}
\newtheorem{theorem}{Theorem}[section]
\newtheorem{proposition}[theorem]{Proposition}
\newtheorem{corollary}[theorem]{Corollary}
\newtheorem{lemma}[theorem]{Lemma}

\theoremstyle{definition}
\newtheorem{definition}[theorem]{Definition}
\newtheorem{remark}[theorem]{Remark}

\begin{document}

\maketitle

\begin{abstract}
  Our study aims to specify the asymptotic error distribution in the
 discretization of a stochastic Volterra equation with a fractional
 kernel. It is well-known that for a standard stochastic differential
 equation, the discretization error, normalized with its rate of
 convergence $1/\sqrt{n}$, converges in law to the solution of a certain
 linear equation. Similarly to this, we show that a suitably normalized
 discretization error of the Volterra equation converges in law to the
 solution of a certain linear Volterra equation with the same fractional
 kernel.
\end{abstract}

\section{Introduction}

The discretization of the stochastic differential equations (SDE) has
been studied by many researchers for many years. 
Let $T>0$ and consider a standard $d$-dimensional SDE and its discretization:
\begin{align*}
  X_t &= X_0 + \int_0^t b(X_s) \mr{d}s + \int_0^t \sigma(X_s) \mr{d}W_s,\quad t \in [0,T],\\
  \hat{X}_t &= X_0 + \int_0^t b(\hat{X}_\frac{[ns]}{n}) \mr{d}s + \int_0^t \sigma(\hat{X}_\frac{[ns]}{n}) \mr{d}W_s,\quad t \in [0,T].
\end{align*}
Back in the 1990s, the limit distribution of the scaled error $\U^{n} =
\sqrt{n} (X -\hat{X})$ was studied by Kurtz and
Protter~\cite{KurtzProtterWong} and Jacod and
Protter~\cite{JacodProtterAsymp}.  
They proved that $\U^{n}$ stably converges in law to $\U$ as $n$ tends to infinity, where $\U = (U^1,\dots,U^d)$ is the solution of the following SDE:
\begin{equation}\label{eqlimitSDE}
  U_t^i = \sum_{k=1}^d \int_0^t U^k_s\left(\partial_k b^i (X_s)   \mr{d}s +
   \sum_{j=1}^m 
\partial_k \sigma^i_j (X_s)   \mr{d} W^j_s\right) - \frac{1}{\sqrt{2}}
   \sum_{j=1}^m \sum_{k=1}^d \sum_{l=1}^m \int_0^t   \partial_k
   \sigma^i_j (X_s) \sigma^k_l(X_s) \mr{d}B^{l,j}_s,
\end{equation}
where $B$ is an $m^2$-dimensional standard Brownian motion independent
of the $m$-dimensional standard Brownian motion $W$.
Recent developments include extensions to stochastic time
partitions~\cite{Fukasawa2020} and SDEs driven by a fractional Brownian
motion~\cite{Hu-Liu-Nualart, Liu-Tindel, Aida-Naganuma}.
Applications include the optimal choice of tuning parameters for the
Multi-Level Monte Carlo method~\cite{Kebaier2015}.

The aim of our study is to extend their result to $d$-dimensional stochastic Volterra equations (SVE) of the form:
\begin{equation}\label{eqSVE}
  X_t = X_0 +  \int_0^t K(t-s) b(X_s) \mr{d}s+ \int_0^t K(t-s)\sigma(X_s) \mr{d} W_s, \quad t \in [0,T],
\end{equation}
where $K(t)=\frac{t^{H-1/2}}{\Gamma(H+1/2)},\,H \in (0,1/2]$. 
Because of the singularity at the origin of the kernel function $K$, 
a sample path of the solution $X$ with $H<1/2$ has lower \hd regularity 
than that of a SDE has.
In particular, the solution $X$ is not a semimartingale.
Recently such a SVE has attracted attention in mathematical finance in
the context of rough volatility modeling; see Abi~Jaber et al.~\cite{AVP} and
references therein. Applications in financial practice, such as pricing
path-dependent options, require numerical
methods to simulate the solution of the SVE. The discretization of the
SVE is then the most natural step. An estimation of the associated
numerical error is therefore of practical importance.

As in the case of SDEs, let $\hat{X}$ be the solution of the discretized SVE of (\ref{eqSVE}), that is,
\begin{equation}\label{eqapproxSVE}
  \hat{X}_t = X_0 + \int_0^t K(t-s) b(\hat{X}_{\frac{[ns]}{n}}) \mr{d} s+\int_0^t K(t-s)\sigma(\hat{X}_{\frac{[ns]}{n}}) \mr{d} W_s, \quad t \in [0,T].
\end{equation}
The solution $\hat{X}_t$ of (\ref{eqapproxSVE}) is constructed in an inductive way. Indeed, we have
\begin{equation*}\label{eqhatXinduction}
  \begin{alignedat}{4}
    \hat{X}_{\frac{1}{n}} &= X_0 + b(X_0)& \int_0^\frac{1}{n} &K(\tfrac{1}{n}-s)  \mr{d}s + \sigma(X_0)\int_0^{\frac{1}{n}} K(\tfrac{1}{n}-s) \mr{d}W_s,\\
    \hat{X}_{\frac{2}{n}} &= X_0 + b(X_0) &\int_0^{\frac{1}{n}} &K(\tfrac{2}{n}-s)  \mr{d}s +b(\hat{X}_{\frac{1}{n}}) \int_{\frac{1}{n}}^{\frac{2}{n}} K(\tfrac{2}{n}-s) \mr{d}s \\
    &&&+\sigma(X_0) \int_0^{\frac{1}{n}} K(\tfrac{2}{n}-s)  \mr{d}W_s +  \sigma(\hat{X}_{\frac{1}{n}}) \int_{\frac{1}{n}}^{\frac{2}{n}} K(\tfrac{2}{n}-s) \mr{d}W_s.
  \end{alignedat}
\end{equation*}
Iterating this procedure, we consequently obtain the solution:
\begin{equation}\label{eqhatXiteration}
    \hat{X}_t = X_0 + \sum_{j=0}^{[nt]} b(\hat{X}_{\frac{j}{n}}) \int_{\frac{j}{n}}^{\frac{j+1}{n}\wedge t} K(t-s)\mr{d} s +\sum_{j=0}^{[nt]} \sigma(\hat{X}_{\frac{j}{n}}) \int_{\frac{j}{n}}^{\frac{j+1}{n}\wedge t} K(t-s)\mr{d} W_s.
\end{equation}
In particular, we can construct $(\hat{X}_{t_1},\dots,\hat{X}_{t_n})$ 
on the regular grid $t_k = \frac{k}{n}$ by generating
the independent sequence of Gaussian vectors
(matrices when the dimension of $W$ is larger than two)
\begin{equation}\label{gauss}
\left(\, \int_{t_{k-1}}^{t_k} K(t_k-s) \mr{d}W_s,
 \,\int_{t_{k-1}}^{t_k} K(t_{k+1}-s) \mr{d}W_s, \dots,
 \, \int_{t_{k-1}}^{t_k} K(t_n-s) \mr{d}W_s\right), \ \ k= 1,2,\dots,n
\end{equation}
using the Cholesky decomposition of the identical covariance matrix
\begin{equation*}
 \Sigma = [\Sigma_{ij}], \ \ \Sigma_{ij}
= \int_{0}^{\frac{1}{n}} K\left(\frac{i}{n}-s\right)K\left(\frac{j}{n}-s\right) \mr{d}s.
\end{equation*}
Therefore \eqref{eqapproxSVE} describes a feasible numerical scheme.
Further, as noted by Fukasawa and Hirano~\cite{Fukasawa-Hirano}, the matrix
$\Sigma$ is nearly degenerate, meaning that the projection
to a low (but more than three when $H\neq 1/2$) dimensional Gaussian vector gives an
accurate and efficient approximation to \eqref{gauss}. 
The extreme approximation is the one dimensional projection that essentially
corresponds to the Euler scheme for SVE, for which Richard et
al.~\cite{Richard-Tan-Yang2021} gave the rate of convergence.
The discretization scheme \eqref{eqapproxSVE} is expected to be more accurate
than the Euler scheme in the sense that it gives a lower mean squared
error (with the same rate of convergence). 
We determine the asymptotic error distribution for the scheme
\eqref{eqapproxSVE} in this paper.
The corresponding analysis for the
Euler scheme is however remained for future research.

The Jacod~\cite{Jacod1997}  theory of stable convergence for
semimartingales played a key role in the study of the discretization
error in the SDE case. Since a solution of a SVE with a singular
fractional kernel is not a semimartingale any more, 
the argument of Jacod and Protter~\cite{JacodProtterAsymp}
is not directly extended to the case of SVE. To overcome this technical
difficulty, we exploit the fact that the convolution with respect to
the fractional kernel $K$ is a continuous map between \hd spaces.
This fact is utilized recently by Horvath et al.~\cite{Horvath} to show the
convergence of random walk approximations to rough volatility models.
We therefore formulate our limit theorem in terms of the weak
convergence of laws on  \hd spaces. 
A tightness criterion for a separable subset of the space of \hd
functions is studied by Hamadouche~\cite{Hamadouche2000}.  
We give, similar but not the same, a useful criterion for the tightness
in the full  space of the \hd functions in Appendix.

We give our main result in Section~\ref{secmain}. The proofs for some lemmas are
deferred to  Section~\ref{secproof}, after some
preliminaries are given in Section~\ref{secpre}.

\section{The main result}\label{secmain}

\subsection{The statement}
Let $(\Omega, \msc{F}, \mathsf{P}, \{\msc{F}_t\}_{t\geq0})$ be a
filtered probability space satisfying the usual conditions, and $W$ be an
$m$-dimensional standard Brownian motion defined on this space.  
Assume the 
coefficients $b:\Real^{d}\to \Real^{d}$ and $\sigma:\Real^{d} \to
\Real^{d\times m}$ in (\ref{eqSVE}) and (\ref{eqapproxSVE}) are
continuously differentiable functions with the derivatives being bounded
and uniformly continuous.\footnote{
We remark here that the uniform continuity of the derivatives would be
relaxed to the  simple continuity assumption by a localization argument:
stopping the processes at the time when the process is going out of some
compact set. Similarly, if we know a priori that $X$ stays in a
domain of $\mathbb{R}^d$, then the regularity conditions on
$\mathbb{R}^d$ would be relaxed to those on the domain.}

We denote by $\mcl{C}_0$ the set of the $\Real^d$-valued continuous
functions on $[0,T]$ vanishing at $t=0$ and by $\mcl{C}^\lambda_0$ the
set of the $\Real^d$-valued $\lambda$-\hd continuous functions with the
same property. Also, $\|\cdot\|_{\infty}, \|\cdot\|_{\mcl{C}^\lambda_0},
\|\cdot\|_{L_p}$ denote the supremum norm on $[0,T]$, the \hd norm on
$[0,T]$, and the $L_p$ norm with respect to $\mathsf{P}$,
respectively. Throughout the paper,  we always view elements of $\Real^{d}$ as column vectors, and denote by $A^\top$ the transpose of $A$ for any matrix $A$. We also denote by $C$ a constant which may differ from one place to another one.

The following theorem is our main result.
\begin{theorem}\label{thmlimdistSVE}
  Let $\epsilon \in (0,H)$. Then the process $\U^{n} = n^{H}(X -
 \hat{X})$ stably converges in law  in $\mcl{C}^{H-\epsilon}_0$ to a
 process $\U = (U^1,\dots,U^d)$ which is the unique continuous solution of the SVE
  \begin{multline}\label{eqlimitSVE}
    U^i_t = \sum_{k=1}^d \int_0^t K(t-s) U^k_s \left(\partial_k b^i(X_s)
   \mr{d}s + \sum_{j=1}^m  \partial_k \sigma^i_j(X_s) \mr{d}W^j_s\right) \\
    - \frac{1}{\sqrt{\Gamma(2H+2)\sin \pi H}}  \sum_{j=1}^m \sum_{k=1}^d
   \sum_{l=1}^m \int_0^t K(t-s)\partial_k \sigma^i_j(X_s)
   \sigma^k_l(X_s)\mr{d}B^{l,j}_s, \, t\in[0,T], \, i = 1,\dots,d,
  \end{multline}
where $B$ is an $m^2$-dimensional standard Brownian motion, independent
 of $\msc{F}$ and defined on some extension of $(\Omega, \msc{F},
 \mathsf{P})$.
\end{theorem}

Note that a stable convergence in law is stronger than a convergence in
law
and weaker than a convergence in probability. See Jacod and
Protter~\cite{Jacod2011} and 
Haeusler and Luschgy~\cite{Haeusler} for the details.
Note also that, since the inclusion $\mcl{C}^{H-\epsilon}_0 \to \mcl{C}_0$ is
continuous, the continuous mapping theorem implies  
the convergence in law of $\U^n$ to $\U$ in $\mcl{C}_0$.
When $H=1/2$, we recover the classical result \eqref{eqlimitSDE}.

An interesting observation from Theorem~\ref{thmlimdistSVE} is that 
when $H$ is small, the discretization suffers from not only a small
rate of convergence but also a large limit variance due to the factor
$\sin \pi H$ in the denominator of \eqref{eqlimitSVE}.
As in the classical case \eqref{eqlimitSDE}, the limit law $\U_t$ is
Gaussian conditionally on $\msc{F}$.

\subsection{Lemmas}
Here we list key steps to prove Theorem~\ref{thmlimdistSVE}
as lemmas,
for which the proofs are deferred to Section~\ref{secproof}. 
We start with observing the
following decomposition of  $\U^n = (U^{n,1},\dots,U^{n,d})$: 
\begin{multline}\label{eqorgSVE}
  U^{n,i}_t \approx  \int_0^t K(t-s) \left(\nabla b^i(\hat{X}_{\frac{[ns]}{n}})^\top \U^n_s \mr{d}s  + \sum_{j=1}^m \nabla\sigma^i_j(\hat{X}_{\frac{[ns]}{n}})^\top \U^n_s \mr{d}W^j_s\right) \\
  +  \int_0^t K(t-s)n^H  \nabla b^i(\hat{X}_{\frac{[ns]}{n}})^\top (\hat{X}_s - \hat{X}_{\frac{[ns]}{n}}) \mr{d}s+ \sum_{j=1}^m \sum_{k=1}^d\int_0^t K(t-s) \partial_k \sigma^i_j(\hat{X}_{\frac{[ns]}{n}})  \mr{d}V^{n,k,j}_s,
\end{multline}
where $\hat{X}=(\hat{X}^1,\dots,\hat{X}^d)$ is the solution of
\eqref{eqapproxSVE} and
$\V^{n} = \{V^{n,k,j}\}$ is defined as
\begin{equation*}
V^{n,k,j}= n^H \int_0^\cdot (\hat{X}^k_s -
 \hat{X}^k_{\frac{[ns]}{n}})\mr{d}W^j_s,\quad 1\leq k\leq d, \  1\leq j\leq m.
\end{equation*}
The first lemma gives the limits of the quadratic variation and
covariation of $\V^{n}$ and $W$.

\renewcommand{\labelenumi}{(\roman{enumi})}
\begin{lemma}\label{lemvnvnvnwquadlim}
  For all $t\in[0,T] , (k_1,k_2) \in \{1,...,d\}^2, 1\leq j \leq m$,
  \begin{enumerate}
    \item
    \begin{equation*}
      \langle V^{n,k_1,j},V^{n,k_2,j} \rangle_t \xrightarrow[\text{in } L_1]{n\to\infty}  \frac{1}{\Gamma(2H+2)\sin \pi H} \sum_{l=1}^m \int_0^t \sigma^{k_1}_l(X_s)\sigma^{k_2}_l(X_s) \mr{d} s ,
    \end{equation*}
    \item
    \begin{equation*}
      \langle V^{n,k,j},W^j \rangle_t \xrightarrow[\text{in } L_1]{n\to\infty} 0.
    \end{equation*}
  \end{enumerate}
\end{lemma}
Note that $\langle V^{n,k_1,i},V^{n,k_2,j} \rangle = 0$ and $\langle
V^{n,k_1,i},W^j \rangle = 0$ for $(k_1,k_2) \in \{1,\dots,d\}^2,
(i,j)\in\{1,\dots,m\}^2, i\neq j$. Then Lemma \ref{lemvnvnvnwquadlim}
and the results of Jacod~\cite{Jacod1997} lead us to specify the limit
distribution of $\V^{n}$ in Lemma \ref{lemvnlawconv}. 
\begin{lemma}\label{lemvnlawconv}
  The process $\V^{n}$ stably converges in law in $\mcl{C}_0$ to a
 continuous process $\V = \{V^{k,j}\}$ of the following form:
  \begin{equation*}
    V^{k,j} =  \frac{1}{\sqrt{\Gamma(2H+2)\sin \pi H}}
     \sum_{l=1}^m\int_0^\cdot \sigma^k_l(X_s) \mr{d}B^{l,j}_s,\ \ 
1\leq k\leq d, \  1\leq j\leq m.
  \end{equation*}
  where $B$ is an $m^2$-dimensional standard Brownian motion, independent of $\msc{F}$ and defined on some extension of $(\Omega, \msc{F}, \mathsf{P})$.
\end{lemma}

We also show that the second integral term of (\ref{eqorgSVE}) vanishes in $\mcl{C}^{H-\epsilon}_0$ as $n$ goes to infinity.
\begin{lemma}\label{lemsecondtermzero}
  For all $i\in \{1,\dots,d\}$,
  \begin{equation*}
    \int_0^t K(t-s)n^H  \nabla b^i(\hat{X}_{\frac{[ns]}{n}})^\top (\hat{X}_s - \hat{X}_{\frac{[ns]}{n}}) \mr{d}s\xrightarrow[\text{in }\mathsf{P} \text{ in } \mcl{C}^{H-\epsilon}_0]{n\to\infty} 0.
  \end{equation*}
\end{lemma}

The difference between both sides of (\ref{eqorgSVE}) converges to zero in $\mcl{C}^{H-\epsilon}_0$ for any $\epsilon\in(0,H)$ as $n$ goes to infinity as we prove in Lemma \ref{lemdiffvanish}.
\begin{lemma}\label{lemdiffvanish}
  The $\gamma$-\hd norm of the difference between both sides of
 (\ref{eqorgSVE}) tends to zero in $L_p$ for any $\gamma \in (0,H)$ and $p\geq 1$.
\end{lemma}

Denote by $\mathcal{D}_d$ the space of the cadlag functions on $[0,T]$
taking values in $\mathbb{R}^d$ equipped with the Skorokhod topology. 
Define $\varphi_n : [0,T] \to [0,T]$ by $\varphi_n(t) = [nt]/n$.
The above lemmas are used in the proof of Lemma~\ref{lemchar}.

\begin{lemma}\label{lemchar}
 If the sequence $$(\U^n,\V^n,\{ \nabla b^i(\hat{X}\circ \varphi_n)\}_i,
 \{\partial_k\sigma^i_j(\hat{X}\circ \varphi_n)\}_{ijk})$$ converges in law in 
$\mathcal{C}^{H-\epsilon}_0 \times \mathcal{C}_0\times
 \mathcal{D}_{d^2} \times \mathcal{D}_{d^2m}$ to
$$(\U,\V,\{ \nabla b^i(X)\}_i,
 \{\partial_k\sigma^i_j(X)\}_{ijk}),$$ 
then $\U$ is the solution of (\ref{eqlimitSVE}).
\end{lemma}

We will additionally show the following lemma.
\begin{lemma}\label{lemtightUn}
The sequence $\U^n$ is tight in $\mcl{C}^{H-\epsilon}_0$ for any $\epsilon\in(0,H)$.
\end{lemma}

We will use the uniqueness in law of the solution of (\ref{eqlimitSVE}).
\begin{lemma}\label{lemsolUeval}
  If there is a strong solution for (\ref{eqlimitSVE}), it is in $L_p$, continuous, and unique in law.
\end{lemma}

\subsection{Proof of Theorem~\ref{thmlimdistSVE}}
Using the above lemmas, we now prove the theorem.
\begin{proof}[Proof of Theorem \ref{thmlimdistSVE}]
By Lemmas~\ref{lemtightUn} and \ref{lemhldcompact}, $\hat{X} \to X $ in probability 
in the uniform topology. Therefore,
$$(\{ \nabla b^i(\hat{X}\circ \varphi_n)\}_i,
 \{\partial_k\sigma^i_j(\hat{X}\circ \varphi_n)\}_{ijk})
\to
(\{ \nabla b^i(X)\}_i,
 \{\partial_k\sigma^i_j(X)\}_{ijk})
$$ 
in probability 
in the uniform topology as well.
Together with Lemmas~\ref{lemvnlawconv} and \ref{lemtightUn}, we
 conclude that
$$(\U^n,\V^n,\{ \nabla b^i(\hat{X}\circ \varphi_n)\}_i,
 \{\partial_k\sigma^i_j(\hat{X}\circ \varphi_n)\}_{ijk}, Y)$$ 
is tight in $\mathcal{C}^{H-\epsilon}_0 \times \mathcal{C}_0\times
 \mathcal{D}_{d^2} \times \mathcal{D}_{d^2m} \times \mathbb{R}$ for any
 random variable $Y$ on $(\Omega,\msc{F},\mathsf{P})$.
For any subsequence of this tight sequence, there exists a further
 subsequence which converges by Prokhorov's theorem (see e.g.,
 Theorem~5.1 of Billingsley~\cite{Billingsley} for a nonseparable case).
 Lemmas~\ref{lemchar}
 and \ref{lemsolUeval} imply the uniqueness of the
 limit. Therefore the original sequence itself has to
 converge. The limit $\U$ of $\U^n$ 
is characterized by \eqref{eqlimitSVE} again by Lemma~\ref{lemchar}.
The convergence of $\U^n$ is stable because $Y$ is arbitrary.
\end{proof}

\section{Preliminaries}\label{secpre}

The fractional kernel satisfies the following condition.
\begin{equation}\label{eqkernelcondition}
  \begin{gathered}
    K \in L_\beta(0,T) \text{ for some } \beta \in(2,\tfrac{2}{1-2H}),\\
    \int_0^h K(t) \mr{d}t = O(h^{H+1/2}),\quad  \int_0^T (K(t+h) - K(t)) \mr{d}t  = O(h^{H+1/2}),\\
    \left(\int_0^h K(t)^2 \mr{d}t\right)^\frac{1}{2} = O(h^H) \text{ and}   \left(\int_0^T (K(t+h) - K(t))^2 \mr{d}t \right)^\frac{1}{2} = O(h^H).
  \end{gathered}
\end{equation}

We fix here $\beta \in(2,\tfrac{2}{1-2H})$ and denote by $\beta^\ast$ the conjugate index of $\beta/2$, namely, $\beta^\ast = \beta/(\beta-2)$, for the technical purposes.

Before discussing the moments and \hd continuities, we remark the existence and uniqueness for the solutions of (\ref{eqSVE}) and (\ref{eqapproxSVE}).
\begin{remark}
  The existence and uniqueness of the strong continuous solution $X$ for
 (\ref{eqSVE}) are guaranteed by Abi Jaber et al.(Theorem 3.3
 \cite{AVP}). 
The strong solution $\hat{X}$ uniquely exists by \eqref{eqhatXiteration}.
\end{remark}

We introduce the following lemma which is used several times to evaluate integrals with convolution kernel.
\begin{lemma}\label{lemcalcforlem}
  \renewcommand{\labelenumi}{(\arabic{enumi})}
  The following inequalities hold for any adapted $\Real^d$-valued process $Y$ and $\Real^{d\times m}$-valued process $Z$:
  \begin{enumerate}
    \item for $p\geq 2$, $\displaystyle \EXP{\left| \int_0^t K(t-s) Y_s \mr{d}s\right|^p} \leq C \int_0^t \EXP{|Y_s|^p} \mr{d}s$,
    \item for  $p>2\beta^\ast$, $\displaystyle \EXP{\left| \int_0^t K(t-s) Z_s \mr{d}W_s\right|^p} \leq C \int_0^t \EXP{|Z_s|^p} \mr{d}s$,
    \item for $p\geq1$, \\
    $\displaystyle \EXP{\left|\int_0^t (K(t+h-s) - K(t-s)) Y_s \mr{d}s\right|^p}+ \EXP{\left| \int_t^{t+h} K(t+h-s)Y_s \mr{d}s\right|^p} \leq C h^{(H+1/2) p} \sup_{r\in[0,T]} \EXP{|Y_r|^{p}}$,
    \item for $p\geq2$,\\
    $\displaystyle \EXP{\left|\int_0^t (K(t+h-s) - K(t-s)) Z_s \mr{d}W_s\right|^p} + \EXP{\left| \int_t^{t+h} K(t+h-s)Z_s \mr{d}W_s\right|^p} \leq C h^{H p} \sup_{r\in[0,T]} \EXP{|Z_r|^{p}}$,
  \end{enumerate}
  where $C$ depends only on any of $K, \beta, p$, and $T$.
\end{lemma}

\begin{proof}[Proof of Lemma \ref{lemcalcforlem}]
  Let us show (1),(3) first. Take $p\geq 2$. Minkowski's integral inequality, the Cauchy-Schwarz inequality and \hds inequality show that
  \begin{align*}
    \EXP{\left| \int_0^t K(t-s) Y_s \mr{d}s\right|^p} &\leq \left(\int_0^t |K(t-s)|\EXP{|Y_s|^p}^\frac{1}{p}\mr{d}s\right)^p\\
    &\leq \left(\int_0^t |K(s)|^2\mr{d}s\right)^\frac{p}{2}\left(\int_0^t \EXP{|Y_s|^p}^\frac{2}{p}\mr{d}s\right)^\frac{p}{2}\\
    &\leq C \int_0^t \EXP{|Y_s|^p} \mr{d}s.
  \end{align*}
  For (3), we observe
  \begin{align*}
    &\EXP{\left|\int_0^t (K(t+h-s) - K(t-s)) Y_s\mr{d}s\right|^p} + \EXP{\left| \int_t^{t+h} K(t+h-s)Y_s \mr{d}s\right|^p} \\
    &\leq \left(\int_0^t |K(t+h-s) - K(t-s)| \EXP{|Y_s|^p}^\frac{1}{p} \mr{d}s\right)^p + \left( \int_t^{t+h}|K(t+h-s)| \EXP{|Y_s|^p}^\frac{1}{p} \mr{d}s\right)^p \\
    &\leq  \sup_{r\in[0,T]}\EXP{|Y_r|^{p}}\left(  \left(\int_0^t (K(s+h) - K(s)) \mr{d}s\right)^p + \left( \int_0^{h} K(s) \mr{d}s\right)^p\right) \\
    &\leq C h^{(H+1/2) p} \sup_{r\in[0,T]} \EXP{|Y_r|^{p}}
  \end{align*}
  by Minkowski's integral inequality and (\ref{eqkernelcondition}).
  We next show (2),(4). For (2), let $p>2\beta^\ast$. Then, using the Burkholder-Davis-Gundy (BDG for short) inequality, \hds inequality and Fubini's theorem properly yields that
  \begin{align*}
    \EXP{\left| \int_0^t K(t-s) Z_s \mr{d}W_s\right|^p}&\leq C_p \EXP{\left(\int_0^t |K(t-s)|^2 |Z_s|^2 \mr{d}s\right)^\frac{p}{2}}\\
    &\leq  \left(\int_0^t K(t-s)^{2\cdot \frac{\beta}{2}} \mr{d}s\right)^{\frac{2}{\beta}\cdot \frac{p}{2}} \EXP{\left(\int_0^t |Z_s|^{2\beta^\ast} \mr{d}s\right)^{\frac{1}{\beta^\ast}\cdot\frac{p}{2}}}\\
    &\leq \left(\int_0^t K(s)^\beta\mr{d}s \right)^{\frac{p}{\beta}}\EXP{t^{\frac{p}{2\beta^\ast}-1} \int_0^t|Z_s|^p \mr{d}s}\\
    &\leq C \int_0^t \EXP{|Z_s|^p} \mr{d}s.
  \end{align*}
  For (4), take $p\geq 2$. By the BDG inequality and Minkowski's inequality, we obtain
  \begin{align*}
    &\EXP{\left|\int_0^t (K(t+h-s) - K(t-s)) Z_s\mr{d}W_s\right|^p}+ \EXP{\left| \int_t^{t+h} K(t+h-s)Z_s \mr{d}W_s\right|^p} \\
    &\leq C_p\EXP{\left(\int_0^t |K(t+h-s) - K(t-s)|^2 |Y_s|^2\mr{d}s\right)^\frac{p}{2}}+ C_p\EXP{\left( \int_t^{t+h} |K(t+h-s)|^2 |Y_s|^2 \mr{d}s\right)^\frac{p}{2}}\\
    &\leq C_p \left(\int_0^t (K(t+h-s) - K(t-s))^2 \EXP{|Y_s|^{2\cdot \frac{p}{2}} }^\frac{2}{p} \mr{d}s\right)^\frac{p}{2} + C_p \left(\int_t^{t+h} K(t+h-s)^2 \EXP{|Y_s|^{2\cdot \frac{p}{2}} }^\frac{2}{p} \mr{d}s\right)^\frac{p}{2}\\
    &\leq C_p \left(\left(\int_0^t (K(s+h) - K(s))^2\mr{d}s\right)^\frac{p}{2} +  \left(\int_0^h K(u)^2\mr{d}u\right)^\frac{p}{2}\right) \sup_{r\in[0,T]} \EXP{|Y_r|^{p}} \\
    &\leq C h^{H p} \sup_{r\in[0,T]} \EXP{|Y_r|^{p}}.
  \end{align*}
  This completes the proof.
\end{proof}

The $p$th moment and the \hd continuity of the solution of the standard
SVE, have already been studied. The following two results are
corollaries of Abi Jaber et al. (Lemmas 3.1 and 2.4 \cite{AVP}).
\begin{lemma}\label{lemXpeval}
  Let $p\geq1$, Then,
  \begin{equation*}
    \sup_{t \in [0,T]} \EXP{|X_t|^p} \leq C,
  \end{equation*}
  where $C$ is a constant that only depends on $|X_0|, |b(0)|, |\sigma(0)|, K, p$ and $T$.
\end{lemma}

\begin{lemma}\label{lemXhd}
  Let $p > H^{-1}$. Then
  \begin{equation*}
    \EXP{|X_t-X_s|^p} \leq C|t-s|^{H p} ,\quad t,s \in [0,T]
  \end{equation*}
  and $X$ admits a version which is \hd continuous on $[0,T]$ of any order $\alpha<H-p^{-1}$. Denoting this version again by $X$, one has
  \begin{equation*}
    \EXP{\left(\sup_{0\leq s \leq t \leq T} \frac{|X_t - X_s| }{|t-s|^\alpha}\right)^p} \leq C_\alpha
  \end{equation*}
  for all $\alpha \in [0,H-p^{-1})$, where $C_\alpha$ is a constant. As a
 consequence, we can regard $X$ as a $\mcl{C}^{\alpha}$ valued random
 variable for any $\alpha<H$.
\end{lemma}

The following two lemmas are analogues of the above lemmas.

\begin{lemma}\label{lemhatXpeval}
  Let $p\geq1$, Then,
  \begin{equation*}
    \sup_{t \in [0,T]} \EXP{|\hat{X}_t|^p} \leq C,
  \end{equation*}
  where $C$ is a constant that only depends on $|X_0|, |b(0)|, |\sigma(0)|, K, p, \beta$, and $T$.
\end{lemma}

\begin{proof}
  We will prove only in the case $p> 2\beta^\ast$, which is sufficient for $p\geq1$. Let $\tau_m = \inf \left\{t\geq 0 \,\middle|\, |\hat{X}_t| \geq m \right\} \wedge T$ and observe that
  \begin{equation}\label{eqlocalbehavior}
    |\hat{X}_t|^p 1_{\{t<\tau_m\}} \leq \left| X_0 + \int_0^t K(t-s) b(\hat{X}_{\frac{[ns]}{n}} 1_{\{s < \tau_m\}}) \mr{d}s+ \int_0^t K(t-s) \sigma(\hat{X}_{\frac{[ns]}{n}} 1_{\{s < \tau_m\}}) \mr{d}W_s \right|^p.
  \end{equation}
  Indeed, for $t \geq \tau_m$ the left-hand side is zero while the right-hand side is nonnegative. For $t<\tau_m$ the local property of the stochastic integral implies that the right-hand side is equal to
  \begin{equation*}
    \left|X_0 + \int_0^t K(t-s) b(\hat{X}_{\frac{[ns]}{n}} ) \mr{d}s + \int_0^t K(t-s) \sigma(\hat{X}_{\frac{[ns]}{n}} ) \mr{d}W_s \right|^p.
  \end{equation*}
  Then by (\ref{eqlocalbehavior}), we have
  \begin{align*}
    \EXP{\left|\hat{X}_t\right|^p1_{\{t<\tau_m\}}} \leq 3^{p-1}\left|X_0\right|^p+ 3^{p-1}&\EXP{\left| \int_0^t K(t-s) b(\hat{X}_{\frac{[ns]}{n}} 1_{\{s < \tau_m\}}) \mr{d}s\right|^p}\\
    + &3^{p-1} \EXP{\left| \int_0^t K(t-s) \sigma(\hat{X}_{\frac{[ns]}{n}} 1_{\{s < \tau_m\}}) \mr{d}W_s\right|^p}.
  \end{align*}
  Therefore, by Lemma \ref{lemcalcforlem}-(1),(2) and the Lipschitz condition of $\sigma$, we see
  \begin{align*}
    \EXP{\left|\hat{X}_t\right|^p1_{\{t<\tau_m\}}} &\leq  C_1 + C_2 \int_0^t \EXP{\left( |b(\hat{X}_{\frac{[ns]}{n}})|^p + |\sigma(\hat{X}_{\frac{[ns]}{n}})|^p \right)1_{\{s < \tau_m\}}} \mr{d}s\\
    &\leq C_1 + C_2\int_0^t \EXP{ |\hat{X}_{\frac{[ns]}{n}}|^p 1_{\{\frac{[ns]}{n} < \tau_m\}}}\mr{d}s\\
    &\leq C_1 + C_2\int_0^t \sup_{r\in[0,s]}\EXP{ |\hat{X}_r|^p 1_{\{r < \tau_m\}}}\mr{d}s,
  \end{align*}
  where $C_1,C_2 \geq 0$ are some constants independent of $n$ and $m$ (remark that $\{s < \tau_m\} \subset \{\frac{[ns]}{n} < \tau_m\}$ for all $s\in (0,T)$). Putting $f_m(t) = \sup_{r\in[0,t]}\EXP{ |\hat{X}_r|^p 1_{\{r < \tau_m\}}}$, we see
  \begin{equation*}
    f_m(t) \leq C_1 + C_2 \sup_{r\in[0,t]} \int_0^r \sup_{u\in[0,s]}\EXP{ |\hat{X}_u|^p 1_{\{u < \tau_m\}}}\mr{d}s = C_1 + C_2 \int_0^t f_m(s) \mr{d}s .
  \end{equation*}
  Since $f_m(t)$ is bounded and, therefore, integrable on $[0,T]$, we can apply Gronwall's lemma to obtain
  \begin{equation*}
    f_m(t) \leq C_1e^{C_2 t} \leq C_1e^{C_2T}.
  \end{equation*}
  By Fatou's lemma, we have
  \begin{equation*}
    \EXP{|\hat{X}_t|^p} = \EXP{\liminf_{m\to\infty} |\hat{X}_t|^p 1_{\{t < \tau_m\}}} \leq \liminf_{m\to\infty} \EXP{|\hat{X}_t|^p 1_{\{t < \tau_m\}}} \leq \liminf_{m\to\infty} f_m(t) \leq C_1e^{C_2 T}
  \end{equation*}
  for all $t\in[0,T]$, which completes the proof.
\end{proof}

\begin{lemma}\label{lemhatXhd}
  Let $p > H^{-1}$. Then
  \begin{equation*}
    \EXP{|\hat{X}_t-\hat{X}_s|^p} \leq C|t-s|^{H p} ,\quad t,s \in [0,T]
  \end{equation*}
  and $\hat{X}$ admits a version which is \hd continuous on $[0,T]$ of any order $\alpha<H-p^{-1}$. Denoting this version again by $\hat{X}$, one has
  \begin{equation*}
    \EXP{\left(\sup_{0\leq s \leq t \leq T} \frac{|\hat{X}_t - \hat{X}_s| }{|t-s|^\alpha}\right)^p} \leq C_\alpha
  \end{equation*}
  for all $\alpha \in [0,H-p^{-1})$, where $C_\alpha$ is a constant that
 does not depend on $n$. As a consequence, we can regard $\hat{X}$ as 
a $\mcl{C}^\alpha$ valued random variable for any order $\alpha<H$ for all $n$.
\end{lemma}

\begin{proof}
  Since
  \begin{equation*}
    \sup_{t\in[0,T]]} \EXP{|b(\hat{X}_t)|^p} + \sup_{t\in[0,T]]}\EXP{|\sigma(\hat{X}_t)|^p} \leq C
  \end{equation*}
  with $C$ independent of $n$ by Lemma \ref{lemhatXpeval}, applying Lemma 2.4 in \cite{AVP} yields
  \begin{equation*}
    \EXP{|\hat{X}_t-\hat{X}_s|^p} \leq C|t-s|^{H p}, \quad t,s \in [0,T],
  \end{equation*}
  and thus, the result follows.
\end{proof}

\begin{lemma}\label{lemdiffpeval}
  Let $p\geq 1$. Then the process $X_t - \hat{X}_t$ uniformly converges to zero in $L_p$ with the rate $n^{-H p}$ as $n$ goes to infinity, that is,
  \begin{equation*}
    \sup_{t\in [0,T] }\EXP{|X_t - \hat{X}_t|^p} \leq C n^{-H p},
  \end{equation*}
  where $C$ is a positive constant which does not depend on $n$.
\end{lemma}

\begin{proof}
  We start with the case $p>2\beta^\ast$. First, we decompose the error as the following:
  \begin{align*}
    |X_t - \hat{X}_t|^p &= \left|\int_0^t K(t-s) (b(X_s) - b(\hat{X}_{\frac{[ns]}{n}}))\mr{d} s + \int_0^t K(t-s) (\sigma(X_s) - \sigma(\hat{X}_{\frac{[ns]}{n}}))\mr{d} W_s\right|^p \\
    &\leq 4^{p-1} \left( \left| \int_0^t K(t-s) (b(X_s) -b(\hat{X}_s)) \mr{d} s\right|^p+\left| \int_0^t K(t-s) (b(\hat{X}_s) - b(\hat{X}_{\frac{[ns]}{n}}))\mr{d} s\right|^p\right)\\
    &\hspace{1em}+4^{p-1} \left( \left| \int_0^t K(t-s) (\sigma(X_s) -\sigma(\hat{X}_s)) \mr{d} W_s\right|^p+\left| \int_0^t K(t-s) (\sigma(\hat{X}_s) - \sigma(\hat{X}_{\frac{[ns]}{n}}))\mr{d} W_s\right|^p\right)\\
    &=: 4^{p-1}(\,(\mr{i}) + (\mr{ii}) + (\mr{iii}) + (\mr{iv})\,)
  \end{align*}
  For (i) and (iii), Lemma \ref{lemcalcforlem}-(1),(2), and the Lipschitz condition yield that
  \begin{align*}
    \EXP{(\mr{i}) + (\mr{iii})} &\leq C \int_0^t \EXP{|b(X_s) -b(\hat{X}_s)|^p}+ \EXP{|\sigma(X_s) -\sigma(\hat{X}_s)|^p}\mr{d} s\\
    &\leq C \int_0^t \EXP{|X_s-\hat{X}_s|^p }\mr{d} s.
  \end{align*}
  By Lemmas \ref{lemcalcforlem}-(1),(2) and \ref{lemhatXhd},  we have for (ii) and (iv),
  \begin{align*}
    \EXP{(\mr{ii})+(\mr{iv})} &\leq C \int_0^t \EXP{|b(\hat{X}_s) -b(\hat{X}_{\frac{[ns]}{n}})|^p} +  \EXP{|\sigma(\hat{X}_s) -\sigma(\hat{X}_{\frac{[ns]}{n}})|^p}\mr{d} s\\
    &\leq C \int_0^t \EXP{|\hat{X}_s-\hat{X}_{\frac{[ns]}{n}}|^p}\mr{d} s\leq Cn^{-H p}.
  \end{align*}
  Hence, putting $f(t) = \EXP{|X_t-\hat{X}_t|^p}$, we obtain
  \begin{equation*}
    f(t) \leq C_1 n^{-H p} + C_2 \int_0^t f(s) \mr{d}s,
  \end{equation*}
  where $C_1, C_2$ are some positive constants independent of $n$ and $t$. By Lemmas \ref{lemXhd} and \ref{lemhatXhd}, $X$ and $\hat{X}$ are both continuous, so $f$ is continuous. Therefore, Gronwall's lemma yields that
  \begin{equation*}
    f(t) \leq C_1n^{-H p} e^{C_2t} \leq C n^{-H p}.
  \end{equation*}
  \par If $1\leq p\leq 2\beta^\ast$, it follows from the concavity that
  \begin{equation*}
    \EXP{|X_t-\hat{X}_t|^p}\leq \EXP{|X_t-\hat{X}_t|^q}^{\frac{p}{q}} \leq Cn^{-H q\cdot\frac{p}{q}} = Cn^{-H p}
  \end{equation*}
  for some $q >2\beta^\ast$. This completes the proof.
\end{proof}

From these estimations, an application of the Garsia-Rodemich-Rumsey
inequality gives the following result
as in Section 4.3.2 of Richard et al.~\cite{Richard-Tan-Yang2021}.
\begin{lemma}\label{lemsupdiffpeval}
  For all $p\geq 1$ and $\epsilon\in(0,H)$, there exists a constant $C>0$ which does not depend on $n$ such that
  \begin{equation*}
    \EXP{\sup_{t\in[0,T]}|X_t - \hat{X}_t|^p } \leq C n^{- p (H -\epsilon)} .
  \end{equation*}
\end{lemma}

\section{Proofs for Lemmas~2.2-8}\label{secproof}
We first show the following auxiliary lemma which plays a key role to prove Lemma \ref{lemvnvnvnwquadlim}.
\begin{lemma}\label{lemintkerneldiffertime}
  For $v < s$,
  \begin{equation*}
    n^{2H}\int_0^{\frac{[nv]}{n}}\left((s-u)^{H-1/2}-(\tfrac{[ns]}{n}-u)^{H-1/2} \right)\left((v-u)^{H-1/2}-(\tfrac{[nv]}{n}-u)^{H-1/2} \right) \mr{d}u \rightarrow 0.
  \end{equation*}
\end{lemma}

\begin{proof}
The claim is clear when $H = 1/2$.
Therefore we only need to consider the case $H < 1/2$.
  Put $\delta_{(n,v)} = v- \tfrac{[nv]}{n}$. 
Then, it follows
  \begin{equation*}
    \begin{split}
      &\hspace{1em}n^{2H}\int_0^{\frac{[nv]}{n}}\left((s-u)^{H-1/2}-(\tfrac{[ns]}{n}-u)^{H-1/2} \right)\left((v-u)^{H-1/2}-(\tfrac{[nv]}{n}-u)^{H-1/2} \right)  \mr{d}u \\
      &\leq n^{2H}\int_0^{\frac{[nv]}{n}}\left((y+s-\tfrac{[nv]}{n})^{H-1/2}-(y+\tfrac{[ns]}{n}-\tfrac{[nv]}{n})^{H-1/2} \right)\left((y+\delta_{(n,v)})^{H-1/2}-y^{H-1/2} \right) \mr{d} y \\
      &=(n\delta_{(n,v)})^{2H}\int_0^{\frac{[nv]}{n\delta_{(n,v)} }}\eta_n(z)  \mr{d} z \leq \int_0^\infty 1_{[0,\frac{[nv]}{n\delta_{(n,v)}}]}(z)\eta_n(z) \mr{d} z,
    \end{split}
  \end{equation*}
  where
  \begin{equation*}
    \eta_n(z)
     =\left(\left(z+\frac{s-[nv]/n}{v-[nv]/n}\right)^{H-1/2}-\left(z+\frac{[ns]/n-[nv]/n}{v-[nv]/n}\right)^{H-1/2}
      \right)\left((z+1)^{H-1/2}-z^{H-1/2} \right) 
  \end{equation*}
when $\delta_{(n,v)} > 0$ and $\eta_n(z) = 0$ otherwise.
  By the triangle inequality,
  \begin{equation*}
    \left|\left(z+\frac{s-[nv]/n}{v-[nv]/n}\right)^{H-1/2}-\left(z+\frac{[ns]/n-[nv]/n}{v-[nv]/n}\right)^{H-1/2}\right| \leq 2 \left(z+\frac{[ns]/n-[nv]/n}{v-[nv]/n}\right)^{H-1/2} \leq 2z^{H-1/2}
  \end{equation*}
  holds for all $z \in (0,\infty)$ when $\delta_{(n,v)} > 0$, so we have
  \begin{equation*}
    |1_{[0,\frac{[nv]}{n\delta_{(n,v)}}]}(z)\eta_n(z)| \leq 2 z^{H-1/2} \left(z^{H-1/2} - (z+1)^{H-1/2}\right).
  \end{equation*}
  Then, the evaluation
  \begin{equation*}
    \int_0^\infty z^{H-1/2} \left(z^{H-1/2} - (z+1)^{H-1/2} \right)\mr{d}z \leq \int_0^1 4z^{2H-1} \mr{d}z + \int_1^\infty 2 z^{H-1/2}z^{H-3/2} \mr{d}z < \infty
  \end{equation*}
  enables us to applying the dominated convergence theorem (DCT for short) that leads to
  \begin{equation}\label{eqksukvuevallim}
    \lim_{n\to \infty} \int_0^\infty 1_{[0,\frac{[nv]}{n\delta_{(n,v)}}]}(z)\eta_n(z) \mr{d} z = 0
  \end{equation}
  since $\eta_n(z) \rightarrow (0-0)((z+1)^{H-1/2}-z^{H-1/2} ) =  0$ for $v<s$.
\end{proof}

\subsection{Proof of Lemma \ref{lemvnvnvnwquadlim}-(i)}
Here we compute the limit of $\langle V^{n,k_1,i},V^{n,k_2,i} \rangle_t$.
Remind that
\begin{equation}\label{eqdecompdeltahatX}
  \begin{split}
    \hat{X}^k_s - \hat{X}^k_{\frac{[ns]}{n}} =& \int_0^{\frac{[ns]}{n}} \left(K(s-u)-K({\tfrac{[ns]}{n}}-u) \right) b^k(\hat{X}_{\frac{[nu]}{n}}) \mr{d} u + b^k (\hat{X}_{\frac{[ns]}{n}}) \int_{{\frac{[ns]}{n}}}^s K\left(s-u\right) \mr{d} u\\
    &+ \sum_{j=1}^m\int_0^{\frac{[ns]}{n}} \left(K(s-u)-K({\tfrac{[ns]}{n}}-u) \right) \sigma^k_j (\hat{X}_{\frac{[nu]}{n}}) \mr{d} W^j_u +  \sum_{j=1}^m\int_{{\frac{[ns]}{n}}}^s K\left(s-u\right) \sigma^k_j (\hat{X}_{\frac{[ns]}{n}}) \mr{d} W^j_u
  \end{split}
\end{equation}
and put the sums of the former two and latter two terms of the right-hand as $\psi^{n,k}_{1,s}, \psi^{n,k}_{2,s}$ respectively. Then we have
\begin{equation}\label{eqn2hhatXsnsn}
  \begin{split}
    \langle V^{n,k_1,i},V^{n,k_2,i} \rangle_t  &= \int_0^t n^{2H} (\hat{X}^{k_1}_s - \hat{X}^{k_1}_{\frac{[ns]}{n}})(\hat{X}^{k_2}_s - \hat{X}^{k_2}_{\frac{[ns]}{n}}) \mr{d}s\\
    &= \int_0^t n^{2H} (\psi^{n,k_1}_{1,s}(\hat{X}^{k_2}_s - \hat{X}^{k_2}_{\frac{[ns]}{n}}) +\psi^{n,k_1}_{2,s}\psi^{n,k_2}_{1,s}+\psi^{n,k_1}_{2,s}\psi^{n,k_2}_{2,s} ) \mr{d}s.
  \end{split}
\end{equation}
Define $\delta_{(n,s)} = s-\frac{[ns]}{n}$.  We first introduce the following evaluations:
\begin{equation}\label{eqKnsntosint}
  \begin{split}
    \int_{\frac{[ns]}{n}}^s K(s-u)^2 \mr{d}u &=  \int_0^{s-\frac{[ns]}{n}} \frac{t^{2H-1}}{G}  \mr{d}u= \frac{1}{2HG}\delta_{(n,s)}^{2H} \leq Cn^{-2H},\\
    \int_0^{\frac{[ns]}{n} } (K(s-u)-K(\tfrac{[ns]}{n}-u) )^2  \mr{d} u &= C \int_0^{\frac{[ns]}{n}} |\mu(u,s-\tfrac{[ns]}{n})|^2 \mr{d} u \\
    &=\delta_{(n,s)}^{2H} \int_0^{\frac{[ns]}{n\delta_{(n,s)}}} |\mu(r,1)|^2  \mr{d} r \\
    &\leq C\delta_{(n,s)}^{2H} \leq Cn^{-2H},
  \end{split}
\end{equation}
where $G=\Gamma(H+1/2)^2$ and $\mu(r,y)=(r+y)^{H-1/2}-r^{H-1/2}$.  We now show $n^{H}\psi^{n,k}_{1,s}$ vanishes in $L_2$ for $k\in\{k_1,k_2\}$ and any $s$.
By Minkowski's integral inequality, (\ref{eqkernelcondition}), and Lemma \ref{lemhatXpeval}, we have
\begin{align*}
  \EXP{|\psi^{n,k}_{1,s}|^2 } &\leq 2\EXP{\left|\int_0^{\frac{[ns]}{n}} \left(K(s-u)-K({\tfrac{[ns]}{n}}-u) \right) b^k(\hat{X}_{\frac{[nu]}{n}}) \mr{d} u\right|^2} + 2\EXP{\left|b^k (\hat{X}_{\frac{[ns]}{n}})\int_{{\frac{[ns]}{n}}}^s K\left(s-u\right) \mr{d} u\right|^2}\\
  &\leq 2\left(\int_0^{\frac{[ns]}{n}} \left(K(s-u)-K({\tfrac{[ns]}{n}}-u) \right) \EXP{|b^k(\hat{X}_{\frac{[nu]}{n}})|^2}^\frac{1}{2} \mr{d} u\right)^2+2\EXP{|b^k (\hat{X}_{\frac{[ns]}{n}})|^2}\left(\int_{{\frac{[ns]}{n}}}^s K\left(s-u\right) \mr{d} u\right)^2\\
  &\leq 2\sup_{r\in[0,T]}\EXP{|b^k(\hat{X}_r)|^2} \left\{\left(\int_0^{\frac{[ns]}{n}} \left(K(u+s-\tfrac{[ns]}{n})-K(u) \right) \mr{d} u\right)^2+\left(\int_0^{s-{\frac{[ns]}{n}}} K(u) \mr{d} u\right)^2\right\}\\
  &\leq C  n^{-(2H+1)},
\end{align*}
which implies $n^{H}\psi^{n,k}_{1,s}\to0$ in $L_2$ uniformly in $s$. Doing the same as the proof of Lemma \ref{lemcalcforlem}-(4) and using Lemma \ref{lemhatXpeval}, we have
\begin{align*}
  \EXP{n^{2H} |\psi^{n,k}_{2,s}|^2} &\leq 2n^{2H} \EXP{\left| \int_0^{\frac{[ns]}{n}} \left(K(s-u)-K({\tfrac{[ns]}{n}}-u) \right) \sigma^k (\hat{X}_{\frac{[nu]}{n}})^\top \mr{d} W_u\right|^2 + \left|\int_{{\frac{[ns]}{n}}}^s K\left(s-u\right) \sigma^k (\hat{X}_{\frac{[nu]}{n}})^\top \mr{d} W_u\right|^2}\\
  &\leq C n^{2H}\delta_{(n,s)}^{2H} \sup_{r\in[0,T]}\EXP{ |\sigma^k (\hat{X}_r)|^2} \leq C
\end{align*}
with $C$ being independent of $n$ and $s$, so by the Cauchy-Schwarz inequality and the bounded convergence theorem (BCT for short), we observe
\begin{equation}\label{eqpsi1tendzero}
  \begin{split}
    &\hspace{1em}\EXP{n^{2H} \left|\int_0^t (\psi^{n,k_1}_{1,s}(\hat{X}^{k_2}_s - \hat{X}^{k_2}_{\frac{[ns]}{n}}) +\psi^{n,k_1}_{2,s}\psi^{n,k_2}_{1,s})  \mr{d}s\right| }\\
    &\leq  \int_0^t \EXP{n^{2H}|\psi^{n,k_1}_{1,s}|^2}^\frac{1}{2} \EXP{n^{2H}|\hat{X}^{k_2}_s - \hat{X}^{k_2}_{\frac{[ns]}{n}}|^2}^\frac{1}{2} \mr{d}s+ \int_0^t \EXP{n^{2H}|\psi^{n,k_1}_{1,s}|^2}^\frac{1}{2}\EXP{n^{2H}|\psi^{n,k_2}_{2,s}|^2}^\frac{1}{2} \mr{d}s \xrightarrow{n\to\infty} 0
  \end{split}
\end{equation}

It becomes sufficient that we only consider the last term on the right-hand side of (\ref{eqn2hhatXsnsn}). We start with decomposing as follows:
\begin{align*}
  &n^{2H}\int_0^t \psi^{n,k_1}_{2,s}\psi^{n,k_2}_{2,s} \mr{d}s\\
  =&\sum_{j,l}^m n^{2H} \int_0^t \left(\int_0^{\frac{[ns]}{n}} f_{k_1j}(s,u) \mr{d} W^j_u\right)\left(\int_0^{\frac{[ns]}{n}} f_{k_2l}(s,u) \mr{d} W^l_u\right) \mr{d}s\\
  &+ \sum_{j,l}^m n^{2H}\int_0^t \sigma^{k_2}_l(\hat{X}_{\frac{[ns]}{n}})\left(\int_0^{\frac{[ns]}{n}} f_{k_1j}(s,u) \mr{d} W^j_u\right)\left(\int_{\frac{[ns]}{n}}^s K(s-u) \mr{d} W^l_u\right) \mr{d} s\\
  &+ \sum_{j,l}^m n^{2H}\int_0^t \sigma^{k_1}_j(\hat{X}_{\frac{[ns]}{n}})\left(\int_0^{\frac{[ns]}{n}} f_{k_2l}(s,u) \mr{d} W^l_u\right)\left(\int_{\frac{[ns]}{n}}^s K(s-u) \mr{d} W^j_u\right) \mr{d} s\\
  &+ \sum_{j,l}^m n^{2H} \int_0^t \sigma^{k_1}_j(\hat{X}_{\frac{[ns]}{n}})\sigma^{k_2}_l(\hat{X}_{\frac{[ns]}{n}})\left(\int_{\frac{[ns]}{n}}^s K(s-u) \mr{d} W^j_u\right)\left(\int_{\frac{[ns]}{n}}^s K(s-u) \mr{d} W^l_u\right) \mr{d}s\\
  =:& \sum_{j=1}^m \sum_{l=1}^m (\mathbf{I}_{jl} + \mathbf{II}_{jl} + \mathbf{III}_{jl} + \mathbf{IV}_{jl}),
\end{align*}
where
\begin{equation*}
  f_{kj}(s,u)=\left(K(s-u)-K(\tfrac{[ns]}{n}-u) \right)\sigma^{k}_j(\hat{X}_{\frac{[nu]}{n}}).
\end{equation*}
We use the following equality which is derived from It\^ o's formula: for any progressively measurable square integrable function $h_1,h_2$,
\begin{equation}\label{eqitoformulti}
  \begin{split}
    &\left(\int_s^t h_1(u) \mr{d} W^j_u\right)\left(\int_s^t h_2(u) \mr{d} W^l_u\right) \\
    &\hspace{2em}= \int_s^t \left(\int_s^u h_1(r) \mr{d} W^j_r\right) h_2(u) \mr{d}W^l_u+\int_s^t \left(\int_s^u h_2(r) \mr{d} W^l_r\right) h_1(u) \mr{d}W^j_u + \int_s^t h_1(u)h_2(u) \mr{d} \langle W^j,W^l \rangle_u.
  \end{split}
\end{equation}
\textit{For $\mathbf{I}_{jl}$.} According to (\ref{eqitoformulti}), $\mathbf{I}_{jl}$ can be written as
\begin{align*}
  \mathbf{I}_{jl} &= n^{2H} \int_0^t \left[ \int_0^{\frac{[ns]}{n}} \left(\int_0^u  f_{k_1j}(s,r) \mr{d} W^j_r\right)  f_{k_2l}(s,u) \mr{d}W^l_u\right.\\
  &\hspace{6em}+ \left.\int_0^{\frac{[ns]}{n}}\left(\int_s^u  f_{k_2l}(s,r) \mr{d} W^l_r\right)  f_{k_1j}(s,u) \mr{d}W^j_u + \int_0^{\frac{[ns]}{n}} f_{k_1j}(s,u)f_{k_2l}(s,u)\mr{d} \langle W^j,W^l \rangle_u\right] \mr{d}s,
\end{align*}
For the last term, it vanishes if $j\neq l$, and if $j=l$, it can be transformed as
\begin{equation}\label{eqfk1jfk2j}
  \begin{split}
    &\hspace{1em}n^{2H} \int_0^t  \int_0^{\frac{[ns]}{n}}  f_{k_1j}(s,u)f_{k_2j}(s,u)\mr{d} u \mr{d}s \\
    &= n^{2H} \int_0^t \int_0^{\frac{[ns]}{n}} \left(K(r+\delta_{(n,s)})-K(r)\right)^2 \sigma^{k_1}_j(\hat{X}_{\frac{[ns]+[-nr]}{n}}) \sigma^{k_2}_j(\hat{X}_{\frac{[ns]+[-nr]}{n}})\mr{d} r \mr{d}s\\
    &= \int_0^t (n\delta_{(n,s)})^{2H}\frac{1}{G} \int_0^{\frac{[ns]}{n\delta_{(n,s)}}} |\mu(r,1)|^2 \prod_{q\in\{k_1,k_2\}}\sigma^{q}_j(\hat{X}_{\frac{[ns]+[-n\delta_{(n,s)} r]}{n}})\mr{d} r\mr{d}s.
  \end{split}
\end{equation}
We will apply Lemma \ref{lemrefineddelattre} here, so it must be checked that the hypothesis is satisfied. We are to show that
\begin{equation*}
  \int_0^{\frac{[ns]}{n\delta_{(n,s)}}} |\mu(r,1)|^2 \prod_{q\in\{k_1,k_2\}}\sigma^{q}_j(\hat{X}_{\frac{[ns]+[-n\delta_{(n,s)} r]}{n}})\mr{d} r \xrightarrow {L_2(\mr{d}s \otimes \mr{d}\mathsf{P})}\sigma^{k_1}_j(X_s)\sigma^{k_2}_j(X_s) \int_0^{\infty}|\mu(r,1)|^2\mr{d} r .
\end{equation*}
We note that the limit on the right is certainly a continuous function of $s$. It follows from Fubini's theorem and Minkowski's integral inequality that
\begin{equation}\label{eqsigmahatXsigmaXl2}
  \begin{split}
    &\EXP{ \int_0^t \left| \int_0^\infty  |\mu(r,1)|^2 \left(1_{\{(0,\frac{[ns]}{n\delta_{(n,s)}})\}}(r) \prod_{q\in\{k_1,k_2\}}\sigma^{q}_j(\hat{X}_{\frac{[ns]+[-n\delta_{(n,s)} r]}{n}}) - \sigma^{k_1}_j(X_s)\sigma^{k_2}_j(X_s) \right)\mr{d}r \right|^2 \mr{d}s}\\
    \leq& \int_0^t \left(\int_0^\infty   |\mu(r,1)|^2 \left\| 1_{\{(0,\frac{[ns]}{n\delta_{(n,s)}})\}}(r) \prod_{q\in\{k_1,k_2\}}\sigma^{q}_j(\hat{X}_{\frac{[ns]+[-n\delta_{(n,s)} r]}{n}}) - \sigma^{k_1}_j(X_s)\sigma^{k_2}_j(X_s)  \right\|_{L_2}\mr{d}r \right)^2 \mr{d}s.
  \end{split}
\end{equation}
Remark that $\|\cdot\|_{L_p}$ simply means the $L_p$ norm with respect to $\mathsf{P}$. Then, by Minkowski's inequality,
\begin{align*}
  &\left\| 1_{\{(0,\frac{[ns]}{n\delta_{(n,s)}})\}}(r)\sigma^{k_1}_j(\hat{X}_{\frac{[ns]+[-n\delta_{(n,s)} r]}{n}}) \sigma^{k_2}_j(\hat{X}_{\frac{[ns]+[-n\delta_{(n,s)} r]}{n}}) - \sigma^{k_1}_j(X_s)\sigma^{k_2}_j(X_s)  \right\|_{L_2}\\
  &=1_{\{(0,\frac{[ns]}{n\delta_{(n,s)}})\}}(r)  \left\| \sigma^{k_1}_j(\hat{X}_{\frac{[ns]+[-n\delta_{(n,s)} r]}{n}}) \sigma^{k_2}_j(\hat{X}_{\frac{[ns]+[-n\delta_{(n,s)} r]}{n}}) - \sigma^{k_1}_j(X_s)\sigma^{k_2}_j(X_s)\right\|_{L_2}\\
  &\hspace{1em} + 1_{\{(\frac{[ns]}{n\delta_{(n,s)}},\infty)\}}(r)\left\| \sigma^{k_1}_j(X_s)\sigma^{k_2}_j(X_s)  \right\|_{L_2}
\end{align*}
holds, and here the last term vanishes as $n\to\infty$. By the Cauchy-Schwarz inequality, Minkowski's inequality, the Lipschitz continuity of $\sigma$, Lemmas \ref{lemXpeval}, \ref{lemhatXpeval}, \ref{lemhatXhd}, and \ref{lemdiffpeval}, we have
\begin{equation}\label{eqsigmahatXsigmaXsq}
  \begin{split}
    &\left\| \prod_{k\in\{k_1,k_2\}}\sigma^{k}(\hat{X}_{\frac{[ns]+[-n\delta_{(n,s)} r]}{n}}) - \sigma^{k_1}_j(X_s)\sigma^{k_2}_j(X_s)\right\|_{L_2}\\
    & \leq\left\| \sigma^{k_1}_j(\hat{X}_{\frac{[ns]+[-n\delta_{(n,s)} r]}{n}}) \{\sigma^{k_2}_j(\hat{X}_{\frac{[ns]+[-n\delta_{(n,s)} r]}{n}}) - \sigma^{k_2}_j(X_s)\}\right\|_{L_2} +\left\| \{\sigma^{k_1}_j(\hat{X}_{\frac{[ns]+[-n\delta_{(n,s)} r]}{n}}) - \sigma^{k_1}_j(X_s)\}\sigma^{k_2}_j(X_s)\right\|_{L_2} \\
    &\leq \left\| \sigma^{k_1}_j(\hat{X}_{\frac{[ns]+[-n\delta_{(n,s)} r]}{n}})\right\|_{L_4}\left\| \sigma^{k_2}_j(\hat{X}_{\frac{[ns]+[-n\delta_{(n,s)} r]}{n}}) - \sigma^{k_2}_j(X_s)\right\|_{L_4} + \left\| \sigma^{k_2}_j(X_s)\right\|_{L_4}\left\| \sigma^{k_1}_j(\hat{X}_{\frac{[ns]+[-n\delta_{(n,s)} r]}{n}}) - \sigma^{k_1}_j(X_s)\right\|_{L_4} \\
    &\leq C \|  \hat{X}_{\frac{[ns]+[-n\delta_{(n,s)} r]}{n}} - \hat{X}_s + \hat{X}_s - X_s \|_{L_4} \leq C ( \| \hat{X}_{\frac{[ns]+[-n\delta_{(n,s)} r]}{n}} - \hat{X}_s\|_{L_4} + \|\hat{X}_s - X_s \|_{L_4}) \\
    &\leq C \left( \left|\tfrac{[ns]+[-n\delta_{(n,s)} r]}{n}- s\right|^H + n^{-H}\right) \leq C n^{-H} \rightarrow 0.
  \end{split}
\end{equation}
Therefore, since
\begin{equation*}
  \frac{[ns]}{n\delta_{(n,s)}} \to \infty, \quad \frac{[ns]+[-n\delta_{(n,s)} r]}{n} \to s,
\end{equation*}
we get
\begin{equation*}
  \left\| 1_{\{(0,\frac{[ns]}{n\delta_{(n,s)}})\}}(r)\sigma^{k_1}_j(\hat{X}_{\frac{[ns]+[-n\delta_{(n,s)} r]}{n}}) \sigma^{k_2}_j(\hat{X}_{\frac{[ns]+[-n\delta_{(n,s)} r]}{n}}) - \sigma^{k_1}_j(X_s)\sigma^{k_2}_j(X_s)  \right\|_{L_2}  \rightarrow 0.
\end{equation*}
Consequently, the right-hand side of (\ref{eqsigmahatXsigmaXl2}) tends to zero by applying the DCT for the integral of $\mr{d}s$ and $\mr{d}r$ respectively, and then, Lemma \ref{lemrefineddelattre} gives the evaluation for (\ref{eqfk1jfk2j}) as follows:
\begin{align*}
  &n^{2H} \int_0^t  \int_0^{\frac{[ns]}{n}}  f_{k_1j}(s,u)f_{k_2j}(s,u)\mr{d} u \mr{d}s \\
  =&\int_0^t (ns-[ns])^{2H}\frac{1}{G} \int_0^{\frac{[ns]}{n\delta_{(n,s)}}} |\mu(r,1)|^2 \prod_{q\in\{k_1,k_2\}}\sigma^{q}_j(\hat{X}_{\frac{[ns]+[-n\delta_{(n,s)} r]}{n}})\mr{d} r\mr{d}s\\
  &\xrightarrow[\text{in } L_2]{n\to \infty}\frac{1}{(2H+1)G} \int_0^\infty |\mu(r,1)|^2 \mr{d} r \int_0^t \sigma^{k_1}_j(X_s)\sigma^{k_2}_j(X_s)  \mr{d}s.
\end{align*}
We next show that the remainder term of $\mathbf{I}_{jl}$,
\begin{equation*}
  I_1^{jl} := n^{2H} \int_0^t  \int_0^{\frac{[ns]}{n}} \left(\int_0^u  f_{k_1j}(s,r) \mr{d} W^j_r\right)  f_{k_2l}(s,u) \mr{d}W^l_u \mr{d}s,
\end{equation*}
converges to zero in $L_2$ as $n\to\infty$. Set
\begin{equation*}
  D^{jl}_{1,s} := n^{2H}\int_0^{\frac{[ns]}{n}} 1_{(0,\frac{[ns]}{n})}(u) \left(\int_0^u  f_{k_1j}(s,r) \mr{d} W^j_r\right)  f_{k_2l}(s,u) \mr{d}W^l_u,
\end{equation*}
and observe that
\begin{align*}
  \EXP{|I_1^{jl}|^2} &=\EXP{ \int_0^t \int_0^t D^{jl}_{1,s} D^{jl}_{1,v} \mr{d}v \mr{d}s}\\
  &=2 \int_0^t \int_0^s \EXP{ D^{jl}_{1,s} D^{jl}_{1,v} } \mr{d} v\mr{d}s
\end{align*}
by Fubini's theorem. Then, by (\ref{eqitoformulti}) and Fubini's theorem, we have
\begin{equation}\label{eqd1snd1vn}
  \begin{split}
    &\EXP{D^{jl}_{1,s}D^{jl}_{1,v}}\\
    =&n^{4H}\EXP{\int_0^{\frac{[nv]}{n}}  \left(\int_0^u  f_{k_1j}(s,r) \mr{d} W^j_r\right) \left(\int_0^u  f_{k_1j}(v,r) \mr{d} W^j_r\right)f_{k_2l}(s,u)f_{k_2l}(v,u) \mr{d}\langle W^j,W^l\rangle_u }\\
    =& \begin{cases}
    0 ,& \text{if } j \neq l,\\
    n^{4H}\int_0^{\frac{[nv]}{n}}  \left(K(s-u)-K(\tfrac{[ns]}{n}-u) \right)\left(K(v-u)-K(\tfrac{[nv]}{n}-u) \right)\tilde{f}(u)\mr{d}u, & \text{if } j=l,
  \end{cases}
\end{split}
\end{equation}
where
\begin{equation*}
  \tilde{f}(u) :=\EXP{|\sigma^{k_2}_j(\hat{X}_{\frac{[nu]}{n}})|^2  \left(\int_0^u  f_{k_1j}(s,r) \mr{d} W^j_r\right) \left(\int_0^u  f_{k_1j}(s,r) \mr{d} W^j_r\right)}.
\end{equation*}
Again by (\ref{eqitoformulti}), it follows that for $u \in [0,\frac{[nv]}{n}]$,
\begin{equation}\label{eqhatx2g1nu}
  \begin{split}
    \left|\tilde{f}(u)\right| \leq& \left|\EXP{|\sigma^{k_2}_j(\hat{X}_{\frac{[nu]}{n}})|^2\int_0^{u}\left(\int_0^{r} f_{k_1j}(s,y) \mr{d} W^j_y\right) f_{k_1j}(v,r)\mr{d}W^j_r}\right| \\
    &+\left|\EXP{|\sigma^{k_2}_j(\hat{X}_{\frac{[nu]}{n}})|^2\int_0^u \left(\int_0^{r} f_{k_1j}(v,y) \mr{d} W^j_y\right) f_{k_1j}(s,r) \mr{d} W^j_r}\right| \\
    &+\left|\EXP{|\sigma^{k_2}_j(\hat{X}_{\frac{[nu]}{n}})|^2\int_0^{u}f_{k_1j}(s,r) f_{k_1j}(v,y)\mr{d} r}\right|.
  \end{split}
\end{equation}
For the first term on the right-hand side, by using the Cauchy-Schwarz inequality, the It\^ o isometry, Fubini's theorem, Lipschitz continuity of $\sigma$, and Lemma \ref{lemhatXpeval} properly, we see
\begin{equation}\label{eqhatx2g1nu1}
  \begin{split}
    &\hspace{1em}\left|\EXP{|\sigma^{k_2}_j(\hat{X}_{\frac{[nu]}{n}})|^2\int_0^{u}\left(\int_0^{r} f_{k_1j}(s,y) \mr{d} W^j_y\right) f_{k_1j}(v,r)\mr{d}W^j_r}\right| \\
    &\leq\EXP{|\sigma^{k_2}_j(\hat{X}_{\frac{[nu]}{n}})|^4}^{\frac{1}{2}} \EXP{\int_0^{u}\left(\int_0^{r} f_{k_1j}(s,y) \mr{d} W^j_y\right)^2 |f_{k_1j}(v,r)|^2\mr{d}r}^{\frac{1}{2}} \\
    &\leq C\left( \int_0^{u} \left(K(v-r)-K(\tfrac{[nv]}{n}-r)\right)^2\EXP{\left(\int_0^{r}f_{k_1j}(s,y) \mr{d} W^j_y\right)^2 \sigma^{k_1}_j(\hat{X}_{\frac{[nr]}{n}})^2}\mr{d}r \right)^{\frac{1}{2}}
  \end{split}
\end{equation}
for some $C$ which does not depend on $n$. Applying \hds inequality, the BDG inequality, Minkowski's integral inequality, Lemma~\ref{lemhatXpeval}, and (\ref{eqKnsntosint}) properly, we have
\begin{align*}
  \EXP{\left(\int_0^{r}f_{k_1j}(s,y) \mr{d} W^j_y\right)^2 \sigma^{k_1}_j(\hat{X}_{\frac{[nr]}{n}})^2} &\leq \EXP{\left(\int_0^{r}f_{k_1j}(s,y) \mr{d} W^j_y\right)^{\beta} }^\frac{2}{\beta} \EXP{\sigma^{k_1}_j(\hat{X}_{\frac{[nr]}{n}})^{2\beta^\ast} }^\frac{1}{\beta^\ast}\\
  &\leq C \EXP{\left(\int_0^{r}|f_{k_1j}(s,y)|^2 \mr{d} y\right)^{\frac{\beta}{2}} }^\frac{2}{\beta}\\
  &\leq C \int_0^{r}\left(K(s-y)-K(\tfrac{[ns]}{n}-y) \right)^2\EXP{|\sigma^{k_1}_j(\hat{X}_{\frac{[ny]}{n}})|^{\beta}}^\frac{2}{\beta} \mr{d} y \\
  &\leq C \sup_{t\in[0,T]} \EXP{|\sigma^{k_1}_j(\hat{X}_t)|^\beta}^\frac{2}{\beta} \int_0^{r}\left(K(s-y)-K(\tfrac{[ns]}{n}-y) \right)^2  \mr{d} y \\
  &\leq C \int_0^{\frac{[ns]}{n} } \left(K(s-y)-K(\tfrac{[ns]}{n}-y) \right)^2  \mr{d} y \leq Cn^{-2H}.
\end{align*}
Hence, (\ref{eqhatx2g1nu1}) is evaluated as
\begin{equation}\label{eqhatx2g1nu1a}
  \begin{split}
    &\left|\EXP{|\sigma^{k_2}_j(\hat{X}_{\frac{[nu]}{n}})|^2\int_0^{u}\left(\int_0^{r} f_{k_1j}(s,y) \mr{d} W^j_y\right) f_{k_1j}(v,r)\mr{d}W^j_r}\right| \\
    &\leq C n^{-H} \left(\int_0^{\frac{[nv]}{n}} \left(K(v-r)-K(\tfrac{[nv]}{n}-r)\right)^2 \mr{d} r \right)^{\frac{1}{2}}\leq C n^{-2H}.
  \end{split}
\end{equation}
Similarly, we have the same evaluation as (\ref{eqhatx2g1nu1})-(\ref{eqhatx2g1nu1a}) for the second term on the right-hand side of (\ref{eqhatx2g1nu}). For the last term remained in (\ref{eqhatx2g1nu}), by Fubini's theorem we have
\begin{align*}
  &\hspace{1em}\left|\EXP{|\sigma^{k_2}_j(\hat{X}_{\frac{[nu]}{n}})|^2\int_0^{u}f_{k_1j}(s,r) f_{k_1j}(v,y)\mr{d} r}\right| \\
  &=\EXP{|\sigma^{k_2}_j(\hat{X}_{\frac{[nu]}{n}})|^2\int_0^{u}\left(K(v-r)-K(\tfrac{[nv]}{n}-r) \right)\left(K(s-r)-K(\tfrac{[ns]}{n}-r) \right)|\sigma^{k_1}_j(\hat{X}_{\frac{[nr]}{n}})|^2 \mr{d} r}\\
  &=\int_0^{u}\left(K(v-r)-K(\tfrac{[nv]}{n}-r) \right)\left(K(s-r)-K(\tfrac{[ns]}{n}-r) \right)\EXP{|\sigma^{k_2}_j(\hat{X}_{\frac{[nu]}{n}})|^2|\sigma^{k_1}_j(\hat{X}_{\frac{[nr]}{n}})|^2} \mr{d} r.
\end{align*}
Since the evaluation
\begin{equation*}
  \EXP{|\sigma^{k_2}_j(\hat{X}_{\frac{[nu]}{n}})|^2|\sigma^{k_1}_j(\hat{X}_{\frac{[nr]}{n}})|^2}
   \leq
   \EXP{|\sigma^{k_2}_j(\hat{X}_{\frac{[nu]}{n}})|^4}^{1/2}\EXP{|\sigma^{k_1}_j(\hat{X}_{\frac{[nr]}{n}})|^4}^{1/2}
   \leq C < \infty
\end{equation*}
is derived from the Cauchy-Schwarz inequality and Lemma \ref{lemhatXpeval}, it follows from the same inequality and (\ref{eqKnsntosint}) that
\begin{align*}
  &\left|\EXP{|\sigma^{k_2}_j(\hat{X}_{\frac{[nu]}{n}})|^2\int_0^{u}f_{k_1j}(s,r) f_{k_1j}(v,y)\mr{d} r}\right| \\
  &\leq C \left(\int_0^{\frac{[nv]}{n}}\left(K(v-r)-K(\tfrac{[nv]}{n}-r) \right)^2 \mr{d} r\right)^{\frac{1}{2}} \left( \int_0^{\frac{[ns]}{n}}\left(K(s-r)-K(\tfrac{[ns]}{n}-r) \right)^2 \mr{d} r \right)^{\frac{1}{2}}\\
  &\leq C n^{-2H}.
\end{align*}
Therefore, $|\tilde{f}(u)| \leq Cn^{-2H}$ is obtained with $C$ being independent of $n$, and it is estimated that for (\ref{eqd1snd1vn}),
\begin{align*}
  \left|\EXP{D^{jl}_{1,s}D^{jl}_{1,v}}\right| &\leq n^{4H}\int_0^{\frac{[nv]}{n}}\left(K(s-u)-K(\tfrac{[ns]}{n}-u) \right)\left(K(v-u)-K(\tfrac{[nv]}{n}-u) \right)  |\tilde{h}_n(u)|  \mr{d}u \\
  &\leq Cn^{2H}\int_0^{\frac{[nv]}{n}}\left(K(s-u)-K(\tfrac{[ns]}{n}-u) \right)\left(K(v-u)-K(\tfrac{[nv]}{n}-u) \right)  \mr{d} u.
\end{align*}
Therefore, Lemma \ref{lemintkerneldiffertime} leads to $|\mathbb{E}[D^{jl}_{1,s}D^{jl}_{1,v}]| \to 0$. Applying the BCT with respect to the integration of $\mr{d}v \otimes \mr{d}s$, we have $I_1^{jl} \xrightarrow[\text{in } L_2]{n\to \infty} 0$, and thus,
\begin{equation*}
  \mathbf{I}_{jl} \xrightarrow[\text{in } L_2]{n\to \infty}
  \begin{cases}
    \frac{1}{(2H+1)G} \int_0^\infty |\mu(r,1)|^2 \mr{d} r \int_0^t \sigma^{k_1}_j(X_s)\sigma^{k_2}_j(X_s)  \mr{d}s, & \text{if }j=l,\\
    0, & \text{if }j\neq l.
  \end{cases}
\end{equation*}
\textit{For $\mathbf{II}_{jl}$.} Write $\mathbf{II}_{jl} = \int_0^t \sigma^{k_2}_l(\hat{X}_{\frac{[ns]}{n}}) D^{jl}_{2,s} \mr{d} s$, where the process $D^{jl}_2$ is the following:
\begin{equation*}
  D^{jl}_{2,s} = n^{2H} \left(\int_0^{\frac{[ns]}{n}} f_{k_1j}(s,u) \mr{d} W^j_u\right)\left(\int_{\frac{[ns]}{n}}^s K(s-u) \mr{d} W^l_u\right)
\end{equation*}
We are to show $\mathbf{II}_{jl}\to 0$ for all $j,l$. Fubini's theorem yields that
\begin{align*}
  \EXP{|\mathbf{II}_{jl}|^2} &=\EXP{ \int_0^t \int_0^t   \sigma^{k_2}_l(\hat{X}_{\frac{[ns]}{n}}) D^{jl}_{2,s}\sigma^{k_2}_l(\hat{X}_{\frac{[nv]}{n}}) D^{jl}_{2,v} \mr{d}v\mr{d}s}\\
  &= \int_0^t 2 \int_0^s \EXP{\sigma^{k_2}_l(\hat{X}_{\frac{[ns]}{n}}) D^{jl}_{2,s}\sigma^{k_2}_l(\hat{X}_{\frac{[nv]}{n}}) D^{jl}_{2,v}} \mr{d}v\mr{d}s\\
  &= 2\int_0^t  \int_0^{\frac{[ns]}{n}} \EXP{\sigma^{k_2}_l(\hat{X}_{\frac{[ns]}{n}}) D^{jl}_{2,s}\sigma^{k_2}_l(\hat{X}_{\frac{[nv]}{n}}) D^{jl}_{2,v}} \mr{d}v\mr{d}s \\
  &\hspace{2em}+2 \int_0^t \int_{\frac{[ns]}{n}}^s \EXP{\sigma^{k_2}_l(\hat{X}_{\frac{[ns]}{n}}) D^{jl}_{2,s}\sigma^{k_2}_l(\hat{X}_{\frac{[nv]}{n}}) D^{jl}_{2,v}} \mr{d}v\mr{d}s\\
  &=:2(I^{jl}_{2,1}+I^{jl}_{2,2}).
\end{align*}
Then, it follows that for all $n$,
\begin{align*}
  I^{jl}_{2,1} &=\int_0^t  \int_0^{\frac{[ns]}{n}} \EXP{ \sigma^{k_2}_l(\hat{X}_{\frac{[nv]}{n}})  D^{jl}_{2,v}  \sigma^{k_2}_l(\hat{X}_{\frac{[ns]}{n}}) \EXP{D^{jl}_{2,s} \middle| \filt{F}_{\frac{[ns]}{n}} } } \mr{d}v\mr{d}s\\
  &=\int_0^t  \int_0^{\frac{[ns]}{n}} \EXP{\sigma^{k_2}_l(\hat{X}_{\frac{[nv]}{n}}) D^{jl}_{2,v}\sigma^{k_2}_l(\hat{X}_{\frac{[ns]}{n}})n^{2H} \int_0^{\frac{[ns]}{n}} f_{k_1j}(s,u) \mr{d} W^j_u \EXP{\int_{\frac{[ns]}{n}}^s K(s-u) \mr{d}  W^l_u \middle| \filt{F}_{\frac{[ns]}{n}} }} \mr{d}v\mr{d}s\\
  &=0.
\end{align*}
We next show $I^{jl}_{2,2}$ tends to zero. By the Cauchy-Schwarz inequality,
\begin{equation*}
  \EXP{\sigma^{k_2}_l(\hat{X}_{\frac{[ns]}{n}}) D^{jl}_{2,s}\sigma^{k_2}_l(\hat{X}_{\frac{[nv]}{n}}) D^{jl}_{2,v}} \leq \EXP{|\sigma^{k_2}_l(\hat{X}_{\frac{[ns]}{n}}) D^{jl}_{2,s}|^2}^\frac{1}{2} \EXP{|\sigma^{k_2}_l(\hat{X}_{\frac{[nv]}{n}}) D^{jl}_{2,v}|^2}^\frac{1}{2}
\end{equation*}
follows, and by again the same inequality, along with Lemma \ref{lemhatXpeval}, the BDG inequality, and (\ref{eqKnsntosint}), we have the following evaluation:
\begin{align*}
  \EXP{|\sigma^{k_2}_l(\hat{X}_{\frac{[ns]}{n}}) D^{jl}_{2,s}|^2} &\leq \EXP{|\sigma^{k_2}_l(\hat{X}_{\frac{[ns]}{n}})|^4}^{\frac{1}{2}} \EXP{|D^{jl}_{2,s}|^4}^\frac{1}{2}\\
  &\leq C n^{4H} \EXP{\left(\int_0^{\frac{[ns]}{n}} f_{k_1j}(s,u) \mr{d} W^j_u\right)^8}^{\frac{1}{4}}\EXP{\left(\int_{\frac{[ns]}{n}}^s K(s-u) \mr{d} W^l_u\right)^8}^\frac{1}{4}\\
  &\leq C n^{4H} \EXP{\left(\int_0^{\frac{[ns]}{n}} |f_{k_1j}(s,u)|^2 \mr{d} u\right)^4}^{\frac{1}{4}}\EXP{\left(\int_{\frac{[ns]}{n}}^s |K(s-u)|^2 \mr{d} u\right)^4}^\frac{1}{4}\\
  &\leq C n^{4H}\int_0^{\frac{[ns]}{n}} \left(K(s-u)-K(\tfrac{[ns]}{n}-u) \right)^2 \EXP{|\sigma^{k_1}_j(\hat{X}_{\frac{[nu]}{n}})|^8}^\frac{1}{4} \mr{d} u\int_{\frac{[ns]}{n}}^s K(s-u)^2 \mr{d} u\\
  &\leq C \sup_{r\in[0,T]}\EXP{|\sigma^{k_1}_j(\hat{X}_r)|^8}^\frac{1}{4} n^{2H} \int_0^{\frac{[ns]}{n}} \left(K(s-u)-K(\tfrac{[ns]}{n}-u) \right)^2 \mr{d} u\leq C
\end{align*}
with $C$ being independent of $n$ and $s$. Hence,
\begin{equation*}
  \left|\EXP{\sigma^{k_2}_l(\hat{X}_{\frac{[ns]}{n}}) D^{jl}_{2,s} \sigma^{k_2}_l(\hat{X}_{\frac{[nv]}{n}})  D^{jl}_{2,v}}\right| \leq C
\end{equation*}
holds, and therefore, the BCT yields
\begin{equation*}
  \lim_{n\to \infty} I^{jl}_{2,2} = \int_0^t \int_0^t \lim_{n\to \infty} 1_{(\frac{[ns]}{n},s)} (v) \EXP{\sigma^{k_2}_l(\hat{X}_{\frac{[ns]}{n}}) D^{jl}_{2,s} \sigma^{k_2}_l(\hat{X}_{\frac{[nv]}{n}})  D^{jl}_{2,v}} \mr{d}v\mr{d}s = 0,
\end{equation*}
which concludes that $\displaystyle \lim_{n\to \infty}\EXP{|\mathbf{II}_{jl}|^2} = 0$, namely, $\mathbf{II}_{jl} \xrightarrow[\text{in } L_2]{n\to\infty}0$.\\
\textit{For $\mathbf{III}_{jl}$.} Similarly to $\mathbf{II}_{jl}$, it holds that $\mathbf{III}_{jl} \to 0$ in $L_2$.\\
\textit{For $\mathbf{IV}_{jl}$.} According to (\ref{eqitoformulti}), we have
\begin{multline}\label{eqIVjl}
  \mathbf{IV}_{jl}= n^{2H} \int_0^t E^{jl}_s \left[ \int_{\frac{[ns]}{n}}^s \left(\int_{\frac{[ns]}{n}}^u K(s-r) \mr{d} W^j_r\right) K(s-u) \mr{d} W^l_u \right.\\
  \left. \hspace{3em}+\int_{\frac{[ns]}{n}}^s \left(\int_{\frac{[ns]}{n}}^u K(s-r) \mr{d} W^l_r\right) K(s-u) \mr{d} W^j_u+ \int_{\frac{[ns]}{n}}^s |K(s-u)|^2 \mr{d}\langle W^j, W^l \rangle_u \right]\mr{d}s,
\end{multline}
where $E^{jl}_s = \sigma^{k_1}_j(\hat{X}_{\frac{[ns]}{n}})\sigma^{k_2}_l(\hat{X}_{\frac{[ns]}{n}})$.
Then, for the last term, it vanishes if $j\neq l$. Otherwise,
\begin{align*}
  n^{2H} \int_0^t E^{jj}_s   \int_{\frac{[ns]}{n}}^s |K(s-u)|^2 \mr{d}u \mr{d}s &=\int_0^t  E^{jj}_s   \frac{n^{2H} }{2HG}\delta_{(n,s)}^{2H} \mr{d}s \\
  &= \int_0^t E^{jj}_s    \frac{(ns-[ns])^{2H}}{2HG}\mr{d}s
\end{align*}
follows from (\ref{eqKnsntosint}), and therefore, Lemma \ref{lemrefineddelattre} yields
\begin{equation*}
  \int_0^t E^{jj}_s  \frac{(ns-[ns])^{2H}}{2HG}\mr{d}s \xrightarrow[\text{in } L_2 ]{n\to \infty} \frac{1}{2HG(2H+1)}\int_0^t \sigma^{k_1}_j(X_s)\sigma^{k_2}_j(X_s)\mr{d}s
\end{equation*}
since
\begin{equation*}
  \EXP{\int_0^t |\sigma^{k_1}_j(\hat{X}_{\frac{[ns]}{n}})\sigma^{k_2}_j(\hat{X}_{\frac{[ns]}{n}}) - \sigma^{k_1}_j(X_s)\sigma^{k_2}_j(X_s)|^2 \mr{d}s} \leq Cn^{-H} \rightarrow 0
\end{equation*}
is given by a similar calculation to (\ref{eqsigmahatXsigmaXsq}). We now consider the remaining first two terms of (\ref{eqIVjl}). It suffices to show that one of them vanishes as $n\to\infty$ since they are symmetry . Define $I^{jl}_4$ as
\begin{equation*}
  I^{jl}_4 = \int_0^t E^{jl}_s D^{jl}_{4,s} \mr{d} s,
\end{equation*}
where
\begin{equation*}
  D^{jl}_{4,s} = n^{2H} \int_{\frac{[ns]}{n}}^s g_j(s,u) \mr{d} W^l_u,\quad     g_j(s,u) = K(s-u)\int_{\frac{[ns]}{n}}^u K(s-r) \mr{d} W^j_r .
\end{equation*}
Here, Fubini's theorem yields that
\begin{align*}
  \EXP{|I^{jl}_4|^2} &= \EXP{\int _0^t \int_0^t  E^{jl}_s D^{jl}_{4,s}E^{jl}_v  D^{jl}_{4,v} \mr{d} v \mr{d} s}\\
  &= 2 \int _0^t \int_0^{\frac{[ns]}{n}} \EXP{E^{jl}_s D^{jl}_{4,s}E^{jl}_v D^{jl}_{4,v}} \mr{d} v \mr{d} s + 2 \int _0^t \int_{\frac{[ns]}{n}}^s \EXP{E^{jl}_s  D^{jl}_{4,s}E^{jl}_v D^{jl}_{4,v}} \mr{d} v \mr{d} s\\
  &=: 2I^{jl}_{4,1} + 2I^{jl}_{4,2}.
\end{align*}
For $v \leq \frac{[ns]}{n}$, since $E^{jl}_s$ is $\filt{F}_{\frac{[ns]}{n}}$-measurable,  the tower property provides
\begin{align*}
  I^{jl}_{4,1} &= \int _0^t \int_0^{\frac{[ns]}{n}} \EXP{E^{jl}_v  D^{jl}_{4,v} E^{jl}_s  \EXP{D^{jl}_{4,s}\middle| \filt{F}_{\frac{[ns]}{n}} }} \mr{d} v \mr{d} s\\
  &= \int _0^t \int_0^{\frac{[ns]}{n}} \EXP{E^{jl}_v  D^{jl}_{4,v} E^{jl}_s  n^{2H}  \EXP{ \int_{\frac{[ns]}{n}}^s  g_j(s,u) \mr{d} W_u \middle| \filt{F}_\frac{[ns]}{n} }} \mr{d} v \mr{d} s\\
  &=0.
\end{align*}
On the other hand, we can rewrite $I_{4,2}$ as
\begin{equation*}
  I^{jl}_{4,2} =\int _0^t \int_0^t 1_{(\frac{[ns]}{n},s)}(v) \EXP{E^{jl}_s  D^{jl}_{4,s}E^{jl}_v D^{jl}_{4,v}} \mr{d} v \mr{d} s = \int _0^t \int_0^t 1_{(\frac{[ns]}{n},s)}(v) \EXP{|E^{jl}_v|^2  D^{jl}_{4,s}D^{jl}_{4,v}} \mr{d} v \mr{d} s
\end{equation*}
since $\tfrac{[nv]}{n} = \tfrac{[ns]}{n}$ for $v\in(\tfrac{[ns]}{n},s)$. We have
\begin{equation*}
  \left|1_{(\frac{[ns]}{n},s)}(v)\EXP{|E^{jl}_v|^2 D^{jl}_{4,s} D^{jl}_{4,v}}\right| =  1_{(\frac{[ns]}{n},s)}(v)\left|\EXP{|E^{jl}_v|^2\EXP{D^{jl}_{4,s}D^{jl}_{4,v}\middle| \filt{F}_{\frac{[nv]}{n}} } } \right| .
\end{equation*}
Then, it follows from (\ref{eqitoformulti}) that
\begin{align*}
  \EXP{D^{jl}_{4,s} D^{jl}_{4,v}\middle| \filt{F}_{\frac{[nv]}{n}} } =&\, n^{4H}\EXP{\int_{\frac{[ns]}{n}}^s g_j(s,u) \mr{d} W^l_u \int_{\frac{[ns]}{n}}^s 1_{(\frac{[nv]}{n},v)}(u) g_j(v,u) \mr{d} W^l_u\middle|\filt{F}_{\frac{[nv]}{n}} }\\
  =&n^{4H}\int_{\frac{[nv]}{n}}^v \EXP{g_j(s,u)g_j(v,u)\middle| \filt{F}_{\frac{[nv]}{n}} } \mr{d}u.
\end{align*}
The last equal follows from Fubini's theorem. Substituting $g_j$, we have
\begin{align*}
  &\hspace{1em}\EXP{g_j(s,u)g_j(v,u)\middle| \filt{F}_{\frac{[nv]}{n}} } \\
  &= K(s-u) K(v-u)\EXP{\left( \int_{\frac{[nv]}{n}}^u K(s-r) \mr{d} W^l_r\right) \left( \int_{\frac{[nv]}{n}}^u K(v-r) \mr{d} W^l_r\right)\middle| \filt{F}_{\frac{[nv]}{n}} }\\
  &= K(s-u) K(v-u)\int_{\frac{[nv]}{n}}^u K(s-r) K(v-r) \mr{d} r.
\end{align*}
Since $K$ is decreasing, using (\ref{eqKnsntosint}), we obtain
\begin{align*}
  \left|\EXP{D^{jl}_{4,s} D^{jl}_{4,v}\middle| \filt{F}_{\frac{[nv]}{n}} }\right| &= n^{4H} \left|\int_{\frac{[nv]}{n}}^v K(s-u) K(v-u) \int_{\frac{[nv]}{n}}^u K(s-r) K(v-r) \mr{d} r \mr{d}u\right|\\
  &\leq n^{4H} \int_{\frac{[nv]}{n}}^v K(v-u)^2 \int_{\frac{[nv]}{n}}^u K(v-r)^2 \mr{d} r \mr{d}u\\
  &\leq \left(n^{2H} \int_{\frac{[nv]}{n}}^v K(v-u)^2 \mr{d}u \right)^2\leq C,
\end{align*}
and hence, we have
\begin{equation*}
  \left|1_{(\frac{[ns]}{n},s)}(v)\EXP{|E^{jl}_v|^2  D^{jl}_{4,s}E^{jl}_v D^{jl}_{4,v}} \right| \leq C\EXP{|E^{jl}_v|^2}\leq C,
\end{equation*}
with $C$ being independent of $s, v,$ and $n$. Therefore, by the BCT, we get
\begin{equation*}
  |I^{jl}_{4,2}| \leq \int _0^t \int _0^t 1_{(\frac{[ns]}{n},s)}(v)\left|\EXP{|E^{jl}_v|^2  D^{jl}_{4,s}E^{jl}_v D^{jl}_{4,v}} \right| \mr{d} v \mr{d} s \xrightarrow{n\to\infty} 0.
\end{equation*}
By the above evaluations, $I^{jl}_{4}$ converges to zero in $L_2$, and thus,
\begin{equation*}
  \mathbf{IV}_{jl} \xrightarrow[\text{in } L_2]{n\to\infty}
  \begin{cases}
    \frac{1}{2HG(2H+1)}\int_0^t \sigma^{k_1}_j(X_s)\sigma^{k_2}_j(X_s) \mr{d}s& \text{if } j=l,\\
    0,& \text{if } j\neq l.
  \end{cases}
\end{equation*}
Finally, we arrive at
\begin{equation*}
  n^{2H}\int_0^t \psi^{n,k_1}_{2,s}\psi^{n,k_2}_{2,s} \mr{d}s \rightarrow \sum_{j=1}^m\frac{1}{2H(2H+1)G} \left( 2H\int_0^\infty |\mu(r,1)|^2 \mr{d} r +  1 \right)\int_0^t \sigma^{k_1}_j(X_s)\sigma^{k_2}_j(X_s)\mr{d}s.
\end{equation*}
Since
\begin{equation*}
  2H\int_0^\infty |\mu(r,1)|^2 \mr{d} r +1 = \frac{G}{\Gamma(2H)\sin \pi H}
\end{equation*}
holds by Mishura (Theorem 1.3.1 and Lemma A.0.1 in \cite{Mishura}), by taking (\ref{eqn2hhatXsnsn}) and (\ref{eqpsi1tendzero}) into account, we obtain for all $i\in\{1,\dots,m\}$ and $(k_1,k_2)\in\{1,\dots,d\}^2$,
\begin{equation*}
  \langle V^{n,k_1,i},V^{n,k_2,i} \rangle_t \xrightarrow[\text{in }L_1]{n \to \infty} \sum_{j=1}^m \frac{1}{\Gamma(2H+2)\sin \pi H}\int_0^t \sigma^{k_1}_j(X_s)\sigma^{k_2}_j(X_s)\mr{d}s.
\end{equation*}
\qed

\subsection{Proof of Lemma \ref{lemvnvnvnwquadlim}-(ii)}
Here we compute the limit of
\begin{equation*}
  \langle V^{n,k,i},W^i \rangle_t = \int_0^t n^{H} (\hat{X}^k_s - \hat{X}^k_{\frac{[ns]}{n}})\mr{d}s
\end{equation*}
for $i \in \{1,\dots, m\}$.
Write $\Delta \hat{X}_s = \hat{X}_s - \hat{X}_{\frac{[ns]}{n}}$. Then, it follows from Fubini's theorem that
\begin{align*}
  \EXP{|\langle V^{n,k,i},W^i \rangle_t|^2} &= \EXP{n^{2H}\int_0^t \int_0^t \Delta \hat{X}^k_s\Delta \hat{X}^k_v \mr{d}v\mr{d}s} \\
  &= \int_0^t 2\int_0^s n^{2H}\EXP{\Delta \hat{X}^k_s\Delta \hat{X}^k_v }\mr{d}v\mr{d}s \\
  &= 2 \int_0^t \int_0^\frac{[ns]}{n} n^{2H}\EXP{\Delta \hat{X}^k_s\Delta \hat{X}^k_v }\mr{d}v\mr{d}s+2\int_0^t \int_\frac{[ns]}{n}^s n^{2H}\EXP{\Delta \hat{X}^k_s\Delta \hat{X}^k_v }\mr{d}v\mr{d}s.
\end{align*}
Write $h_k(s,u) = K(s-u)-K(\tfrac{[ns]}{n}-u) b^k(\hat{X}_{\frac{[nu]}{n}})$.
Using the same notation $f_{kj}$ as in the above proof, by the tower property and (\ref{eqdecompdeltahatX}), we observe that for $v\in(0,\frac{[ns]}{n}), s\in (0,t)$,
\begin{align*}
  \EXP{\Delta \hat{X}^k_s\Delta \hat{X}^k_v } &= \EXP{\Delta \hat{X}^k_v \EXP{\Delta \hat{X}^k_s \middle| \filt{F}_{\frac{[ns]}{n}} } }\\
  &= \EXP{\Delta \hat{X}^k_v \left(\int_0^\frac{[ns]}{n} h_k(s,u) \mr{d}u +\int_\frac{[ns]}{n}^s K(s-u) b^k(\hat{X}_{\frac{[nu]}{n}}) \mr{d}u \right)}\\
  &\hspace{1em} +\sum_{j=1}^m\EXP{ \Delta \hat{X}^k_v \int_0^{\frac{[ns]}{n}}  f_{kj}(s,u) \mr{d}W^j_u }\\
  &= \EXP{\Delta \hat{X}^k_v \left(\int_0^\frac{[ns]}{n} h_k(s,u) \mr{d}u +\int_\frac{[ns]}{n}^s K(s-u) b^k(\hat{X}_{\frac{[nu]}{n}}) \mr{d}u \right)}\\
  &\hspace{1em}+\sum_{j=1}^m\EXP{ \int_0^{\frac{[ns]}{n}} f_{kj}(s,u) \mr{d}W^j_u\left( \int_0^{\frac{[nv]}{n}} h_k(v,u) \mr{d}u +\int_{\frac{[nv]}{n}}^v K(v-u) b^k(\hat{X}_{\frac{[nu]}{n} })\mr{d}u \right)}\\
  &\hspace{1em}+ \sum_{j=1}^m\sum_{l=1}^m \EXP{ \int_0^{\frac{[ns]}{n}}  f_{kj}(s,u) \mr{d}W^j_u \left( \int_0^{\frac{[nv]}{n}}  f_{kl}(v,u) \mr{d}W^l_u +\int_{\frac{[nv]}{n}}^v K(v-u)\sigma^k_l(\hat{X}_{\frac{[nu]}{n} }) \mr{d}W^l_u \right)}\\
  &=: (\mr{I}) + \sum_{j=1}^m (\mr{II})_j + \sum_{j=1}^m\sum_{l=1}^m(\mr{III})_{jl}
\end{align*}
According to (\ref{eqkernelcondition}), (I) is calculated as
\begin{align*}
  \left| (\mr{I})\right|&\leq  \EXP{\int_0^\frac{[ns]}{n}\left|\Delta \hat{X}^k_v h_k(s,u) \right|\mr{d}u +\int_\frac{[ns]}{n}^s K(s-u)\left| \Delta \hat{X}^k_vb^k(\hat{X}_{\frac{[nu]}{n}}) \right|\mr{d}u}\\
  &\leq  \int_0^\frac{[ns]}{n} (K(\tfrac{[ns]}{n}-u)-K(s-u))\EXP{ \left|\Delta \hat{X}^k_v b^k(\hat{X}_{\frac{[nu]}{n}})\right|} \mr{d}u +\int_\frac{[ns]}{n}^s K(s-u) \EXP{ \left|\Delta \hat{X}^k_v b^k(\hat{X}_{\frac{[nu]}{n}})\right|}\mr{d}u\\
  &\leq Cn^{-H} \left( \int_0^\frac{[ns]}{n} (K(\tfrac{[ns]}{n}-u)-K(s-u)) \mr{d}u +\int_\frac{[ns]}{n}^s K(s-u) \mr{d}u \right) \leq Cn^{-(2H+1/2)}
\end{align*}
since
\begin{equation*}
  \EXP{ \left|\Delta \hat{X}^k_v b^k(\hat{X}_{\frac{[nu]}{n}})\right|} \leq \EXP{|\Delta \hat{X}^k_v|^2}^\frac{1}{2} \EXP{ b^k(\hat{X}_{\frac{[nu]}{n}})^2}^\frac{1}{2} \leq Cn^{-H}
\end{equation*}
is true by Lemmas \ref{lemhatXpeval} and \ref{lemhatXhd} (remark $C$ is independent of $u,v$, and $s$). Hence, this term vanishes as $n\to\infty$ even if it is multiplied by $n^{2H}$. For $(\mr{II})_j$, we can derive the same evaluation as $(\mr{I})$ since
\begin{equation*}
  \mathbb{E}[|\int_0^{\frac{[ns]}{n}} f_{kj}(s,u) \mr{d}W^j_u|^2]^{\frac{1}{2}} \leq \int_0^{\frac{[ns]}{n}}  \mathbb{E}[|f_{kj}(s,u)|^2] \mr{d}u \leq Cn^{-H}
\end{equation*}
holds by the It\^o isometry, so $n^{2H}(\mr{II})_j$ vanishes as $n\to\infty$. What remained to be considered is $(\mr{III})_{jl}$ . For the former part of $(\mr{III})_{jl}$, it follows from (\ref{eqitoformulti}) that
\begin{align*}
  &\hspace{1em}\left|\EXP{\int_0^{\frac{[ns]}{n} } f_{kj}(s,u) \mr{d}W^j_u  \int_0^{\frac{[ns]}{n} }  1_{(0,\frac{[nv]}{n} )}(u) f_{kl}(v,u) \mr{d}W^l_u  }\right|\\
  &= \int_0^{\frac{[nv]}{n}}\left(K(s-u)-K(\tfrac{[ns]}{n}-u) \right)\left(K(v-u)-K(\tfrac{[nv]}{n}-u) \right) \EXP{|\sigma^k_j(\hat{X}_{\frac{[nu]}{n} })\sigma^k_l(\hat{X}_{\frac{[nu]}{n} })|} \mr{d}\langle W^j , W^l \rangle.
\end{align*}
If $j\neq l$, it is zero. If $j=l$, since $\sup_{t\in[0,T]} \EXP{|\sigma^k_j(\hat{X}_t)|^2}  < C$ for some finite $C$ which does not depend on $n$ by Lemma \ref{lemhatXpeval}, it follows from Lemma \ref{lemintkerneldiffertime} that
\begin{equation*}
  n^{2H} \EXP{\int_0^{\frac{[ns]}{n} } f_{kj}(s,u) \mr{d}W^j_u  \int_0^{\frac{[ns]}{n} }  1_{(0,\frac{[nv]}{n} )}(u) f_{kl}(v,u) \mr{d}W^l_u } \rightarrow 0.
\end{equation*}
For the latter part of $(\mr{III})_{jl}$, it holds by (\ref{eqitoformulti}) that
\begin{align*}
  &\EXP{ \int_0^{\frac{[ns]}{n}} f_{kj}(s,u) \mr{d}W^j_u \int_0^{\frac{[ns]}{n}} 1_{(\frac{[nv]}{n},v)}(u) K(v-u)\sigma^k_j(\hat{X}_{\frac{[nu]}{n} }) \mr{d}W^j_u} \\
  =&\EXP{\int_{\frac{[nv]}{n}}^v f_{kj}(s,u)  K(v-u)\sigma^k_l(\hat{X}_{\frac{[nu]}{n} })  \mr{d}\langle W^j, W^l \rangle_u}.
\end{align*}
Then, it becomes zero if $j\neq l$, and if $j=l$, again by Lemma \ref{lemhatXpeval},
\begin{align*}
  \left|\EXP{\int_{\frac{[nv]}{n}}^v  f_{kj}(s,u) K(v-u)\sigma^k_j(\hat{X}_{\frac{[nu]}{n} }) \mr{d}u}\right| &= \EXP{|\sigma^k_j(\hat{X}_{\frac{[nv]}{n} })|^2} \int_{\frac{[nv]}{n}}^v \left(K(s-u)-K(\tfrac{[ns]}{n}-u) \right) K(v-u) \mr{d}u\\
  &\leq C \int_0^{v-\frac{[nv]}{n}} \left((r+s-v)^{H-1/2}-(r+\tfrac{[ns]}{n}-v)^{H-1/2} \right) r^{H-1/2} \mr{d}u\\
  &\leq C n^{-2H} \int_0^1 \left((r+\tfrac{s-v}{\delta_{(n,v)}})^{H-1/2}-(r+\tfrac{[ns]/n-v}{\delta_{(n,v)}})^{H-1/2} \right) r^{H-1/2} \mr{d}u
\end{align*}
holds. Doing similarly to (\ref{eqksukvuevallim}) yields
\begin{equation*}
  \int_0^1 \left((r+\tfrac{s-v}{\delta_{(n,v)}})^{H-1/2}-(r+\tfrac{[ns]/n-v}{\delta_{(n,v)}})^{H-1/2} \right) r^{H-1/2} \mr{d}u \rightarrow 0,
\end{equation*}
and thus, it holds
\begin{equation*}
  n^{2H} \EXP{\int_{\frac{[nv]}{n}}^v  f_{kj}(s,u) K(v-u)\sigma^k_j(\hat{X}_{\frac{[nu]}{n} }) \mr{d}u} \rightarrow 0,
\end{equation*}
which concludes that $n^{2H}(\mr{III})_{jl} \to 0$. Therefore,
\begin{equation*}
  n^{2H}\EXP{\Delta \hat{X}^k_s\Delta \hat{X}^k_v } \xrightarrow{n\to\infty} 0, \quad 0< v < \frac{[ns]}{n}, 0 < s < t.
\end{equation*}
On the other hand, for $v\in(\frac{[ns]}{n},s), s\in(0,t)$, from Lemma \ref{lemhatXhd}, together with the Cauchy-Schwarz inequality,
\begin{equation}\label{eqdelhatXsdelhatXveval}
  \left|\EXP{\Delta \hat{X}^k_s\Delta \hat{X}^k_v }\right| \leq \EXP{|\Delta \hat{X}^k_s|^2}^{\frac{1}{2}}\EXP{|\Delta \hat{X}^k_v|^2}^{\frac{1}{2}} \leq C \left(s-\tfrac{[ns]}{n}\right)^{H}\left(v-\tfrac{[nv]}{n}\right)^{H} \leq Cn^{-2H}
\end{equation}
follows and
\begin{equation*}
  \int_0^t \int_0^t 1_{[\frac{[ns]}{n}, s]} n^{2H}\EXP{\Delta \hat{X}^k_s\Delta \hat{X}^k_v } \xrightarrow{n\to\infty} 0
\end{equation*}
holds by the BCT. Since (\ref{eqdelhatXsdelhatXveval}) is valid for all $s,v\in[0,t]$, applying the BCT with respect to the integral of $\mr{d}v\otimes \mr{d} s$ and concluding above discussions, we obtain for all $n\in\Natu, k\in\{1,\dots,d\}, i\in\{1,\dots,m\}$,
\begin{equation*}
  \EXP{|\langle V^{n,k,i}, W^i \rangle_t|^2} \xrightarrow{n\to\infty} 0.
\end{equation*}
\qed

\subsection{Proof of Lemma \ref{lemvnlawconv}}
By Lemma \ref{lemvnvnvnwquadlim}, we can apply Theorem 4-1 of Jacod
\cite{Jacod1997} to see that $\V^{n}$ stably converges in law in
$\mcl{C}_0$ to a conditionally Gaussian martingale $\V = \{V^{k,j}\}$ with
\begin{equation*}
  \begin{aligned}
    \langle V^{k_1,j_1},V^{k_2,j_2} \rangle_t &= \left\{
    \begin{alignedat}{3}
      &\sum_{l=1}^m \frac{1}{\Gamma(2H+2)\sin \pi H}\int_0^t \sigma^{k_1}_l(X_s)\sigma^{k_2}_l(X_s)\mr{d}s, & \quad \text{if } j_1=j_2,\\
      &0,&\text{if } j_1\neq j_2,
    \end{alignedat}
    \right.\\
    \langle V^{k,i},W^{j} \rangle_t &= 0, \quad \,^\forall k\in \{1,\dots,d\} , \,^\forall (i,j)\in \{1,\dots,m\}^2 .
  \end{aligned}
\end{equation*}
Furthermore, since $\mr{d}\langle V,V \rangle_t \ll \mr{d} t$, Proposition 1-4 of Jacod \cite{Jacod1997} yields that $V$ is represented as
\begin{equation*}
  V^{k,j} = \sum_{l=1}^m \frac{1}{\sqrt{\Gamma(2H+2)\sin \pi H}}\int_0^\cdot \sigma^k_l(X_s)\mr{d}B^{l,j}_s
\end{equation*}
where $B$ is $m^2$-dimensional standard Brownian motion which is independent of $W$. This concludes the proof.
\qed

\subsection{Proof of Lemma \ref{lemsecondtermzero}}
Here we show the convergence to $0$ of
\begin{equation*}
  \int_0^t K(t-s)n^H  \nabla b^i(\hat{X}_{\frac{[ns]}{n}})^\top (\hat{X}_s - \hat{X}_{\frac{[ns]}{n}}) \mr{d}s = \sum_{k=1}^d \int_0^t K(t-s) \partial_k b^i(\hat{X}_{\frac{[ns]}{n}}) \mr{d}\langle V^{n,k,j}, W^j\rangle_s.
\end{equation*}
Since $\langle V^{n,k,j}, W^j\rangle_t$ tends to zero in $L_1$ for all
$t\in[0,T]$ by Lemma \ref{lemvnvnvnwquadlim}, we will obtain the result
by using Theorem 7.10 of Kurtz and Protter \cite{Kurtz1996} and the
continuity of $\mcl{J}^\alpha$ given in Lemma~\ref{lemJalphaconti}. 
We start with showing the tightness in $\mcl{C}_0$ of $\{\langle V^{n,k,j}, W^j\rangle\}_{n\in\Natu}$. By Minkowski's integral inequality and Lemma \ref{lemhatXhd}, we have for $0\leq s<t\leq T$,
\begin{align*}
  \EXP{|\langle V^{n,k,j}, W^j\rangle_t-\langle V^{n,k,j}, W^j\rangle_s|^p}  &\leq n^{Hp} \left(\int_s^t \EXP{|\hat{X}^k_u - \hat{X}^k_{\frac{[nu]}{n}}|^p}^\frac{1}{p} \mr{d}u \right)^p\\
  &\leq C n^{Hp} \left(\int_s^t (u - \tfrac{[nu]}{n})^{H} \mr{d}y \right)^p \leq C |t-s|^p.
\end{align*}
Therefore, Kolmogorov's continuity theorem implies $\mathbb{E}[\|\langle V^{n,k,j}, W^j\rangle\|_{\mcl{C}^{\rho}_0}]\leq C_\rho$ for any $\rho\in(0,1)$ with $C_\rho$ being independent of $n$, and then, $\{\langle V^{n,k,j}, W^j\rangle\}_{n\in\Natu}$ is tight in $\mcl{C}^{\rho}_0$ for any $\rho\in(0,1)$ by Theorem \ref{thmvntight}. Assume an arbitrary subsequence $\{\langle V^{n_q,k,j}, W^j\rangle\}_{q\in\Natu}$ of $\{\langle V^{n,k,j},W\rangle\}_{n\in\Natu}$. Since this subsequence is tight in $\mcl{C}^{\rho}_0$, too, Prokhorov's theorem yields that there exists a further subsequence $\{\langle V^{n'_q,k,j},W^j\rangle\}_{q\in\Natu}$  of $\{\langle V^{n_q,k,j},W^j\rangle\}_{q\in\Natu}$ and a process $F\in \mcl{C}^{\rho}_0$ such that $\langle V^{n'_q,k,j},W^j\rangle$ weakly tends to $F$ in $\mcl{C}^{\rho}_0$. However, $\langle V^{n'_q,k,j},W^j\rangle_t$ tends to zero in $L_1$ by Lemma $\ref{lemvnvnvnwquadlim}$ for all $t$, so $F_t = 0$ for all $t$. Thus, $\langle V^{n,k,j},W^j\rangle \xrightarrow{n\to\infty} 0$ in law in $\mcl{C}^{\rho}_0$ and in $\mcl{C}_0$ as well. This implies $\langle V^{n,k,j},W^j\rangle$ tends to zero also in probability in $\mcl{C}_0$.
Let $\{Y^{n}\}_{n\in\Natu}$ be the set of simple predictable processes almost surely bounded by one. By Lemma \ref{lemhatXhd}, we have
\begin{align*}
  \EXP{\left|\int_0^t Y^{n}_{s-} \mr{d} \langle V^{n,k,j},W^j\rangle_s\right|} \leq \int_0^t n^{H} \EXP{|\hat{X}^k_s - \hat{X}^k_{\frac{[ns]}{n}}|}\mr{d} s \leq C,
\end{align*}
which implies $\langle V^{n,k,j},W^j\rangle$ is uniform tight in the
sense of Definition 7.4 of Kurtz and Protter~\cite{Kurtz1996}. 
Noting also that $\hat{X} - X \to 0$ in probability in $\mcl{C}_0$
by Lemma~\ref{lemsupdiffpeval},
Theorem 7.10 of Kurtz and Protter \cite{Kurtz1996} yields
\begin{equation*}
  ( \partial_k b^i(\hat{X}_\frac{[n\cdot]}{n}), \langle V^{n,k,j},W^j \rangle, \partial_k b^i(\hat{X}_\frac{[n\cdot]}{n}) \cdot \langle V^{n,k,j},W^j \rangle )  \xrightarrow[\text{stably in law in }\mathcal{D}]{n\to\infty} (\partial_k b^i(X), 0 ,  0).
\end{equation*}
This implies $\int_0^\cdot \partial_k b^i(\hat{X}_\frac{[ns]}{n}) \mr{d}\langle V^{n,k,j},W^j \rangle_s $ converges in probability in $\mcl{C}_0$ to zero process. We have also for $0\leq s<t\leq T$,
\begin{align*}
  &\EXP{\left|\sum_{k=1}^d\left(\int_0^t \partial_k b^i(\hat{X}_\frac{[nu]}{n}) \mr{d}\langle V^{n,k,j},W^j\rangle_u - \int_0^s \partial_k b^i(\hat{X}_\frac{[nu]}{n}) \mr{d}\langle V^{n,k,j},W^j\rangle_u\right)\right|^p} \\
  &\leq\left( \int_s^t n^{H} \sum_{k=1}^d \EXP{|\partial_k b^i(\hat{X}_\frac{[nu]}{n})  (\hat{X}^k_u - \hat{X}^k_{\frac{[nu]}{n}})|^p}^\frac{1}{p}\mr{d}u\right)^p \\
  &\leq C \left( \int_s^t n^{H}(u - \tfrac{[nu]}{n})^{H}\mr{d}u\right)^p\leq C |t-s|^p
\end{align*}
by Minkowski's integral inequality, Lemma \ref{lemhatXhd}, and the boundedness of the derivatives of $b$. Thus, Kolmogorov's continuity theorem yields
\begin{equation*}
  \EXP{\left\| \sum_{k=1}^d\int_0^\cdot \partial_k b^i(\hat{X}_\frac{[ns]}{n}) \mr{d}\langle V^{n,k,j},W^j \rangle_u \right\|_{\mcl{C}^{\rho}_0}} \leq C
\end{equation*}
with $C$ being independent of $n$ for any $\rho\in(0,1)$ . Consequently, by Corollary \ref{corvnconvhld}, we have the process $\sum_{k}\int_0^\cdot \partial_k b^i(\hat{X}_\frac{[ns]}{n}) \mr{d}\langle V^{n,k,j},W^j \rangle_u $ converges in $\mcl{C}^{1/2-\epsilon}$ to zero process, and the desired result is obtained by Lemma \ref{lemJalphaconti}.
\qed

\subsection{Proof of Lemma \ref{lemdiffvanish}}
Here we show the convergence to 0 of  $\Delta^{n,i}$ defined by
\begin{multline*}
  \Delta^{n,i}_t = U^{n,i}_t - \int_0^t K(t-s) \left(\nabla b^i(\hat{X}_{\frac{[ns]}{n}})^\top \U^n_s \mr{d}s  + \sum_{j=1}^m \nabla\sigma^i_j(\hat{X}_{\frac{[ns]}{n}})^\top \U^n_s \mr{d}W^j_s\right) \\
  +  \int_0^t K(t-s)n^H  \nabla b^i(\hat{X}_{\frac{[ns]}{n}})^\top (\hat{X}_s - \hat{X}_{\frac{[ns]}{n}}) \mr{d}s+ \sum_{j=1}^m \sum_{k=1}^d\int_0^t K(t-s) \partial_k \sigma^i_j(\hat{X}_{\frac{[ns]}{n}})  \mr{d}V^{n,k,j}_s.
\end{multline*}
By Taylor's theorem, one has the following identity for $\sigma$:
\begin{equation*}
  \sigma(x+h) - \sigma(x) - h \sigma^\prime(x) = h \int_0^1 (\sigma^\prime(x+rh)-\sigma^\prime(x)) \mr{d}r .
\end{equation*}
Using this identity, we can rewrite $\Delta^{n,i}$ as
\begin{align*}
  \Delta^{n,i}_t =& \sum_{k=1}^d \int_0^t K(t-s) n^{H} (X^k_s - \hat{X}^k_s) \left(\int_0^1 \{\partial_k b^i(\hat{X}_s+r(X_s - \hat{X}_s))-\partial_k b^i(\hat{X}_s)\} \mr{d}r\right)\mr{d}s \\
  &+\sum_{k=1}^d \int_0^t K(t-s) n^{H} (\hat{X}^k_s - \hat{X}^k_{\frac{[ns]}{n}}) \left(\int_0^1 \{\partial_k b^i( \hat{X}_{\frac{[ns]}{n}}+r(\hat{X}_s - \hat{X}_{\frac{[ns]}{n}}))-\partial_k b^i( \hat{X}_{\frac{[ns]}{n}})\} \mr{d}r\right)\mr{d}s\\
  &+\sum_{j=1}^m\sum_{k=1}^d\int_0^t K(t-s) n^{H} (X^k_s - \hat{X}^k_s) \left(\int_0^1 \{\partial_k \sigma^i_j(\hat{X}_s+r(X_s - \hat{X}_s))-\partial_k \sigma^i_j(\hat{X}_s)\} \mr{d}r\right)\mr{d}W^j_s \\
  &+\sum_{j=1}^m\sum_{k=1}^d\int_0^t K(t-s) n^{H} (\hat{X}^k_s - \hat{X}^k_{\frac{[ns]}{n}}) \left(\int_0^1 \{\partial_k \sigma^i_j( \hat{X}_{\frac{[ns]}{n}}+r(\hat{X}_s - \hat{X}_{\frac{[ns]}{n}}))-\partial_k \sigma^i_j( \hat{X}_{\frac{[ns]}{n}})\} \mr{d}r\right)\mr{d}W^j_s\\
  =\,& \sum_{k=1}^d \tilde{\Delta}^{n,i,k}_{1,t} (X,\hat{X}) +  \sum_{k=1}^d\tilde{\Delta}^{n,i,k}_{1,t} (\hat{X},\hat{X}_{\frac{[n\cdot]}{n}})  +\sum_{j=1}^m\sum_{k=1}^d \tilde{\Delta}^{n,i,j,k}_{2,t} (X,\hat{X}) +  \sum_{j=1}^m\sum_{k=1}^d\tilde{\Delta}^{n,i,j,k}_{2,t} (\hat{X},\hat{X}_{\frac{[n\cdot]}{n}}) ,
\end{align*}
where
\begin{gather*}
  \tilde{\Delta}^{n,i,k}_{1,t}(x,y) = \int_0^t K(t-s) n^{H} \psi^k_{b^i}(s,x,y) \mr{d}s, \quad \tilde{\Delta}^{n,i,j,k}_{2,t}(x,y) = \int_0^t K(t-s) n^{H} \psi^{k}_{\sigma^i_j}(s,x,y) \mr{d}W_s, \\
  \psi^k_a(s,x,y) = (x^k_s - y^k_s) \int_0^1  \{\partial_k a(y_s+r(x_s - y_s))-\partial_k a(y_s)\} \mr{d}r,\quad a \in \{b^i,\sigma^i_j\} ,
\end{gather*}
for any adapted processes $x, y$.  From Lemma \ref{lemcalcforlem}-(3),(4), it follows that for $t+h,t \in [0,T], h>0, p> 2\beta^\ast$ and any adapted processes $x, y$,
\begin{align*}
  \EXP{|\tilde{\Delta}^{n,i,k}_{1,t+h}(x,y) - \tilde{\Delta}^{n,i,k}_{1,t}(x,y) |^p}  &\leq 2^{p-1}\EXP{\left|\int_0^t (K(t+h-s)-K(t-s)) n^{H} \psi^k_{b^i}(s,x,y)\mr{d}s\right|^p}\\
  &\hspace{1em}+2^{p-1}\EXP{\left|\int_t^{t+h} K(t+h-s) n^{H}\psi^k_{b^i}(s,x,y) \mr{d}s\right|^p}\\
  &\leq   Ch^{Hp}n^{Hp} \sup_{s\in[0,T]}  \EXP{|\psi^k_{b^i}(s,x,y)|^p}
\end{align*}
and
\begin{align*}
  \EXP{|\tilde{\Delta}^{n,i,j,k}_{2,t+h}(x,y) - \tilde{\Delta}^{n,i,j,k}_{2,t}(x,y) |^p}  &\leq 2^{p-1}\EXP{\left|\int_0^t (K(t+h-s)-K(t-s)) n^{H} \psi^{k}_{\sigma^i_j}(s,x,y)  \mr{d}W^j_s\right|^p}\\
  &\hspace{1em}+2^{p-1}\EXP{\left|\int_t^{t+h} K(t+h-s) n^{H} \psi^{k}_{\sigma^i_j}(s,x,y)  \mr{d}W^j_s\right|^p}\\
  &\leq   Ch^{Hp}n^{Hp} \sup_{s\in[0,T]}  \EXP{|\psi^{k}_{\sigma^i_j}(s,x,y) |^p}.
\end{align*}
We here define the modulus of continuity of a continuous function $a$ as follows:
\begin{equation*}
  m(a,\delta) = \sup_{\substack{|u-v|\leq\delta\\ u,v\in[0,T]}} |a(u)-a(v)|.
\end{equation*}
Using this notation, by the Cauchy-Schwarz inequality, we have
\begin{align*}
  \EXP{|\psi^k_a(s,x,y)|^p} &= \EXP{|x^k_s-y^k_s|^p \left|\int_0^1 \{\partial_k a(y_s+r(x_s - y_s))-\partial_k a(y_s)\} \mr{d}r \right|^p}\\
  &\leq \EXP{|x^k_s-y^k_s|^{2p}}^{\frac{1}{2}} \EXP{\left|\int_0^1 m(\partial_k a,r|x_s - y_s|) \mr{d}r \right|^{2p}}^{\frac{1}{2}}\\
  &\leq \EXP{|x_s-y_s|^{2p}}^{\frac{1}{2}} \EXP{|m(\partial_k a,\|x -y\|_\infty)|^{2p}}^{\frac{1}{2}}.
\end{align*}
By the hypothesis, the derivatives of $b$ and $\sigma$ are all bounded, and therefore, if $\|x-y\|_\infty \to 0$ in probability as $n\to\infty$, it follows from the BCT and the property of the modulus of continuity that
\begin{equation*}
  \lim_{n\to\infty} \EXP{|m(\partial_k a,\|x -y\|_\infty)|^{2p}} = \EXP{\lim_{n\to\infty}|m(\partial_k a,\|x -y\|_\infty)|^{2p}}= 0
\end{equation*}
for $a = b^i$ and $a = \sigma^i_j$.

Substituting $(x,y)=(X,\hat{X})$ and $(x,y)=(\hat{X},
\hat{X}_{\frac{[n\cdot]}{n}})$ and using Lemmas \ref{lemdiffpeval} and \ref{lemhatXhd} respectively, we have
\begin{align*}
  &\EXP{|\Delta^{n,i}_{t+h}-\Delta^{n,i}_t|^p} \\
  \leq&\, C\sum_{k=1}^d \EXP{|\tilde{\Delta}^{n,i,k}_{1,t+h}(X,\hat{X}) - \tilde{\Delta}^{n,i,k}_{1,t}(X,\hat{X}) |^p}  + C\sum_{j=1}^m\sum_{k=1}^d\EXP{|\tilde{\Delta}^{n,i,j,k}_{2,t+h}(X,\hat{X}) - \tilde{\Delta}^{n,i,j,k}_{2,t}(X,\hat{X}) |^p} \\
  &+C\sum_{k=1}^d\EXP{|\tilde{\Delta}^{n,i,k}_{1,t+h}(\hat{X}, \hat{X}_{\frac{[n\cdot]}{n}}) - \tilde{\Delta}^{n,i,k}_{1,t}(\hat{X}, \hat{X}_{\frac{[n\cdot]}{n}}) |^p}  + C\sum_{j=1}^m\sum_{k=1}^d\EXP{|\tilde{\Delta}^{n,i,j,k}_{2,t+h}(\hat{X}, \hat{X}_{\frac{[n\cdot]}{n}}) - \tilde{\Delta}^{n,i,j,k}_{2,t}(\hat{X}, \hat{X}_{\frac{[n\cdot]}{n}}) |^p}\\
  \leq&\, Ch^{Hp} n^{Hp} \sup_{s\in[0,T]} \EXP{|X_s-\hat{X}_s|^{2p}}^\frac{1}{2}\sum_{k=1}^d \left(\EXP{|m(\partial_k b^i, |X_s-\hat{X}_s|)|^{2p}}^{\frac{1}{2}} +  \sum_{j=1}^m\EXP{|m(\partial_k \sigma^i_j, |X_s-\hat{X}_s|)|^{2p}}^{\frac{1}{2}} \right)\\
  &+\, Ch^{Hp} n^{Hp} \sup_{s\in[0,T]} \EXP{|\hat{X}_s-\hat{X}_{\frac{[ns]}{n}}|^{2p}}^\frac{1}{2}\sum_{k=1}^d \left(\EXP{|m(\partial_k b^i, |\hat{X}_s-\hat{X}_{\frac{[ns]}{n}}|)|^{2p}}^{\frac{1}{2}} +  \sum_{j=1}^m\EXP{|m(\partial_k \sigma^i_j, |\hat{X}_s-\hat{X}_{\frac{[ns]}{n}}|)|^{2p}}^{\frac{1}{2}}\right)\\
  \leq&\, Ch^{Hp}  \sum_{k=1}^d \left(\EXP{|m(\partial_k b^i, \|X-\hat{X}\|_{\infty})|^{2p}}^{\frac{1}{2}} +  \sum_{j=1}^m\EXP{|m(\partial_k \sigma^i_j, \|\hat{X}-\hat{X}_{\frac{[n\cdot]}{n}}\|_{\infty})|^{2p}}^{\frac{1}{2}}\right),
\end{align*}
where $C$ is independent of $n,t$ and $h$. Therefore, since $\Delta^{n,i}_0=0$ and both $\|X-\hat{X}\|_{\infty}$ and $\|\hat{X}-\hat{X}_{\frac{[n\cdot]}{n}}\|_{\infty}$ tend to zero in probability, the desired result is obtained by
\begin{equation*}
  \EXP{\| \Delta^{n,i} \|_{\mcl{C}^{\gamma}_0}^p } \leq C\sum_{k=1}^d \left(\EXP{|m(\partial_k b^i, \|X-\hat{X}\|_{\infty})|^{2p}}^{\frac{1}{2}} +  \sum_{j=1}^m\EXP{|m(\partial_k \sigma^i_j, \|\hat{X}-\hat{X}_{\frac{[n\cdot]}{n}}\|_{\infty})|^{2p}}^{\frac{1}{2}}\right) \rightarrow 0, \,^\forall \gamma\in(0,H),
\end{equation*}
which is derived from Kolmogorov's continuity theorem; see Revus and Yor (Theorem I.2.1 \cite{RevuzYor}).
\qed

\subsection{Proof of Lemma~\ref{lemchar}}
First we show that $\V^n$ is uniformly tight
in the sense of Definition 7.4 of Kurtz and Protter~\cite{Kurtz1996}.
Let
$\{Y^{n}\}_{n\in\Natu}$ be the set of simple predictable processes
almost surely bounded by one. Then for all $t\in[0,T]$, it follows from
the It\^o isometry,
\begin{equation*}
  \EXP{\left|\int_0^t Y^{n}_{s-} \mr{d}V^{n,k,j}_s\right|^2} = \EXP{\int_0^t |Y^{n}_{s-}|^2 \mr{d}\langle V^{n,k,j}\rangle_s} \leq \EXP{\langle V^{n,k,j}\rangle_t} < \infty,
\end{equation*}
where the bound is uniform in $n$ since $\{\langle V^{n,k,j}
\rangle_t\}_{n\in\Natu}$ is the convergent sequence in $L_1$ and,
therefore, bounded sequence in $L_1$ by Lemma \ref{lemvnvnvnwquadlim}.

Now, define $\mathbf{\Phi}^n = (\Phi^{n,1},\dots,\Phi^{n,d})$ by
  \begin{equation*}
    \Phi^{n,i}_t = \sum_{k=1}^d \int_0^t \partial_k b^i(\hat{X}_{\frac{[ns]}{n}})
     U^{n,k}_s \mr{d}s + \sum_{j=1}^m \sum_{k=1}^d \int_0^t \partial_k
     \sigma^i_j(\hat{X}_{\frac{[ns]}{n}})U^{n,k}_s \mr{d}W^j_s +
\sum_{j=1}^m \sum_{k=1}^d\int_0^\cdot \partial_k
\sigma^i_j(\hat{X}_{\frac{[ns]}{n}})  \mr{d}V^{n,k,j}_s.
  \end{equation*}
By the uniform tightness of $\V^n$, Theorem~7.10 of Kurtz and
Protter~\cite{Kurtz1996} implies 
$(\U^n,\mathbf{\Phi}^n) \to (\U,\mathbf{\Phi})$ in law, where
$\mathbf{\Phi} = (\Phi^1,\dots,\Phi^d)$ are defined by
\begin{equation*}
\Phi^i=
\sum_{k=1}^d \int_0^t \partial_k b^i(X_s)
     U^{k}_s \mr{d}s + \sum_{j=1}^m \sum_{k=1}^d \int_0^t \partial_k
     \sigma^i_j(X_s)U^{k}_s \mr{d}W^j_s +
\sum_{j=1}^m \sum_{k=1}^d\int_0^\cdot \partial_k
\sigma^i_j(X_s)  \mr{d}V^{k,j}_s.
\end{equation*}
We can also show $\mathbf{\Phi}^n$ is tight as a
$\mathcal{C}^{1/2-\epsilon}_0$-valued sequence for any
$\epsilon \in (0,1/2)$. Indeed, denoting by 
\begin{equation}\label{tildeVn}
 \tilde{V}^{n,i}
= \sum_{j=1}^m \sum_{k=1}^d\int_0^\cdot \partial_k
\sigma^i_j(\hat{X}_{\frac{[ns]}{n}})  \mr{d}V^{n,k,j}_s,
\end{equation}
we have for any $p >2$ and $0\leq s<t \leq T$,
  \begin{equation}\label{eqtildevnhd0}
    \begin{split}
      \EXP{\left|\tilde{V}^{n,i}_t-\tilde{V}^{n,i}_s \right|^p} &\leq C_{p,d,m}\sum_{j=1}^m \sum_{k=1}^d\EXP{\left|\int_s^t  \partial_k \sigma^i_j(\hat{X}_\frac{[nu]}{n}) n^{H} (\hat{X}^k_u -\hat{X}^k_\frac{[nu]}{n})\mr{d}W^j_u\right|^p} \\
      &\leq C \sum_{j=1}^m \sum_{k=1}^d\EXP{\left(\int_s^t | \partial_k \sigma^i_j(\hat{X}_\frac{[nu]}{n})|^2 n^{2H} |\hat{X}^k_u -\hat{X}^k_\frac{[nu]}{n}|^2\mr{d}u\right)^\frac{p}{2}} \\
      &\leq C\sum_{j=1}^m \sum_{k=1}^d \left(\int_s^t  n^{2H} \EXP{|\partial_k \sigma^i_j((\hat{X}_\frac{[nu]}{n})|^p |\hat{X}^k_u -\hat{X}^k_\frac{[nu]}{n}|^p}^\frac{2}{p}\mr{d}u\right)^\frac{p}{2} \\
      &\leq C\sum_{j=1}^m \sum_{k=1}^d \left(\int_s^t  n^{2H} (u-\tfrac{[nu]}{n})^{2H}  \mr{d}u\right)^\frac{p}{2} \\
      &\leq C |t-s|^{\frac{p}{2}}.
    \end{split}
  \end{equation}
and so, for $t+h,t\in[0,T], h>0$,
  \begin{align*}
    \left\|\Phi^{n,i}_{t+h} -  \Phi^{n,i}_t\right\|_{L_p} &\leq
   \sum_{k=1}^d  \left\| \int_t^{t+h}  \partial_k b^i(\hat{X}_{\frac{[ns]}{n}})
   U^{n,k}_s\mr{d}s\right\|_{L_p} +\sum_{j,k} \left\| \int_t^{t+h}
   \partial_k \sigma^i_j(\hat{X}_{\frac{[ns]}{n}})U^{n,k}_s \mr{d}W^j_s
   \right\|_{L_p} + \|\tilde{V}^{n,i}_{t+h}-\tilde{V}^{n,i}_t\|_{L_p}\\
    &\leq \sum_{k=1}^d \int_t^{t+h}  \|  \partial_k b^i(\hat{X}_{\frac{[ns]}{n}})
   U^{n,k}_s\|_{L_p} \mr{d}s + C_1\sum_{j,k}\left(\int_t^{t+h}
   \|\partial_k \sigma^i_j(\hat{X}_{\frac{[ns]}{n}})U^{n,k}_s\|^2_{L_p} \mr{d}s
   \right)^\frac{1}{2} + C_2h^{1/2}\\
    &\leq Ch^{1/2}
  \end{align*}
by the BDG inequality, Minkowski's integral inequality,
Lemmas~\ref{lemhatXhd} and \ref{lemdiffpeval}, and the boundedness of
the derivatives of
$\sigma$.
Hence, Kolmogorov's continuity
theorem yields
  $\EXP{\|\Phi^{n,i}\|_{\mcl{C}^{1/2-\epsilon}_0}} \leq C$
  uniformly in $n$, and therefore, by Corollary \ref{corvnconvhld}, 
 $\mathbf{\Phi}^{n}$ converges in law in
  $\mcl{C}^{1/2-\epsilon}_0$ to $\mathbf{\Phi}$.

Let
  $\alpha = \frac{1}{2}-H$ and $\gamma = \frac{1}{2}-\epsilon$. By Lemma
  \ref{lemJalphaconti}, the operator $\mcl{J}^\alpha$ is continuous from
  $\mcl{C}^{\gamma}_0$ into $\mcl{C}^{\gamma-\alpha}_0$. Since $(\U^n, \mathbf{\Phi}^{n})$ 
  converges in law to $(\U,\mathbf{\Phi})$ in $\mcl{C}^{\gamma-\alpha}_0 \times
  \mcl{C}^{\gamma}_0$, the continuous mapping theorem implies that 
  $\U^n - \mcl{J}^\alpha \mathbf{\Phi}^n$ converges in law to
  $\U- \mcl{J}^\alpha \mathbf{\Phi}$ in
  $\mcl{C}^{\gamma-\alpha}_0=\mcl{C}^{H-\epsilon}_0$.
On the other hand, Lemmas~\ref{lemsecondtermzero} and
  \ref{lemdiffvanish} imply 
 $\U^n - \mcl{J}^\alpha \mathbf{\Phi}^n$ converges in law to zero.
Consequently $\U- \mcl{J}^\alpha \mathbf{\Phi} = 0$, which is equivalent to
  \eqref{eqlimitSVE}. 
\qed

\subsection{Proof of Lemma \ref{lemtightUn}}
Set $\tilde{U}^{n,i}$ as
\begin{equation*}
  \tilde{U}^{n,i}_t = \int_0^t K(t-s) \nabla b^i(\hat{X}_{\frac{[ns]}{n}})^\top U^{n,i}_s\mr{d}s + \sum_{k=1}^d\sum_{j=1}^m\int_0^t K(t-s) U^{n,i}_s\partial_k \sigma^i_j(\hat{X}_{\frac{[ns]}{n}}) \mr{d}W^j_s.
\end{equation*}
We first show the tightness of $\{\tilde{U}^{n,i}\}$. 
By Theorem~\ref{thmvntight}, it suffices to show that there exists a
uniform bound for
$\mathbb{E}[\|\tilde{U}^{n,i}_t\|_{\mcl{C}^{H-\epsilon^\prime}_0}]$ for
$\epsilon^\prime \in (0,\epsilon)$. To show the \hd continuity of
$\tilde{U}^{n,i}$, we start with decomposing the amount of its change
from $t$ to $t+h$ for some $h>0$ and $p\geq 2$ as the following:
\begin{equation}\label{equnhld}
  \begin{split}
    |\tilde{U}^{n,i}_{t+h} -\tilde{U}^{n,i}_t|^p 
    &\leq 4^{p-1}\Biggl( \left|\int_0^t (K(t+h-s)-K(t-s)) \nabla
   b^i(\hat{X}_{\frac{[ns]}{n}})^\top  U^{n,i}_s \mr{d}s\right|^p \\
    &\hspace{2em} + \left|\int_t^{t+h}
   K(t+h-s)  \nabla b^i(\hat{X}_{\frac{[ns]}{n}})^\top  U^{n,i}_s
   \mr{d}s\right|^p\\
    &\hspace{2em}+ \left|\sum_{k=1}^d\sum_{j=1}^m\int_0^t
   (K(t+h-s)-K(t-s)) \partial_k \sigma^i_j(\hat{X}_{\frac{[ns]}{n}})
   U^{n,i}_s\mr{d}W^j_s\right|^p  \\
    &\hspace{2em} +
   \left|\sum_{k=1}^d\sum_{j=1}^m\int_t^{t+h} K(t+h-s) \partial_k
   \sigma^i_j(\hat{X}_{\frac{[ns]}{n}})U^{n,i}_s\mr{d}W^j_s\right|^p\Biggr).
  \end{split}
\end{equation}
Then, by Lemma \ref{lemcalcforlem}-(3),(4), we have
\begin{equation*}
  \EXP{ |\tilde{U}^{n,i}_{t+h} -\tilde{U}^{n,i}_t|^p} \leq C h^{Hp}  \sup_{t\in[0,T]} \EXP{|U^{n,i}_t|^p} = C h^{Hp}  \sup_{t\in[0,T]} \EXP{n^{Hp} |X_t-\hat{X}_t|^p} \leq Ch^{Hp},
\end{equation*}
where $C$ is independent of $t$ and $n$, and hence, Kolmogorov's continuity theorem and Theorem \ref{thmvntight} leads to the tightness in $\mcl{C}^{H-\epsilon}_0$ of $\{\tilde{U}^{n,i}\}$.

In light of Lemmas~\ref{lemsecondtermzero} and
  \ref{lemdiffvanish},
it is only remained to show that $\{\int_0^\cdot K(\cdot-s)
\mr{d}\tilde{V}^{n,i}_s\}_{n\in\Natu}$ is tight, where $\tilde{V}^{n,i}$
is defined by (\ref{tildeVn}). Similarly to the above, for
any $t+h,t \in [0,T], h>0, p >2$,
\begin{align*}
  &\EXP{\left|\int_0^{t+h} K(t+h-s) \mr{d}\tilde{V}^{n,i}_s - \int_0^t K(t-s) \mr{d}\tilde{V}^{n,i}_s \right|^p} \\
  &\leq 2^{p-1}\EXP{\left|\int_0^t (K(t+h-s)-K(t-s))  \mr{d}\tilde{V}^{n,i}_s \right|^p + \left|\int_t^{t+h} K(t+h-s)  \mr{d}\tilde{V}^{n,i}_s \right|^p}\\
  &\leq 2^{p-1}\EXP{\left|\int_0^t (K(t+h-s)-K(t-s))   \sum_{j,k}
 \partial_k \sigma^i_j(\hat{X}_{\frac{[ns]}{n}})  n^H (\hat{X}^k_s -
 \hat{X}^k_{\frac{[ns]}{n}}) \mr{d}W^j_s \right|^p}\\
  &\hspace{1em} + 2^{p-1}\EXP{\left|\int_t^{t+h} K(t+h-s) \sum_{j,k} \partial_k \sigma^i_j(\hat{X}_{\frac{[ns]}{n}})  n^H (\hat{X}^k_s - \hat{X}^k_{\frac{[ns]}{n}}) \mr{d}W^j_s \right|^p}\\
  &\leq Ch^{Hp}n^{Hp} \sum_{k=1}^d \sup_{r\in[0,T]} \EXP{ |\hat{X}^k_r-\hat{X}^k_\frac{[nr]}{n}|^p} \leq Ch^{Hp}
\end{align*}
follows from Lemmas \ref{lemcalcforlem}-(4), \ref{lemhatXhd}, and the boundedness for the derivative of $\sigma$. Then, by Kolmogorov's continuity theorem, together with Theorem \ref{thmvntight}, we obtain the tightness in $\mcl{C}^{H-\epsilon}_0$ of $\{\int_0^\cdot K(\cdot-s) \mr{d}\tilde{V}^{n,i}_s\}$. From these arguments, the tightness in $\mcl{C}^{H-\epsilon}_0$ of $\{U^{n,i}\}$ is verified.
\qed

\subsection{Proof of Lemma \ref{lemsolUeval}}
We first show the $L_p$ integrability with the similar way to  the proof of Lemma \ref{lemhatXpeval}. The SVE (\ref{eqlimitSVE}) is transformed as
\begin{equation*}
  U^i_t =  \int_0^t K(t-s) \nabla b^i(X_s)^\top U_s \mr{d}s + \sum_{j=1}^m\sum_{k=1}^d \sum_{l=1}^m \int_0^t K(t-s) \partial_k \sigma^i_j(X_s) \{U_s^k\mr{d}W^j_s + C_H\sigma^k_l(X_s)\mr{d}B^{l,j}_s\},
\end{equation*}
where $C_H=\frac{1}{\sqrt{\Gamma(2H+2)\sin \pi H}}$.
Let $\tau_m = \inf\{t\mid |U_t|\geq m \}$. Using the local property, Lemmas \ref{lemcalcforlem}-(1),(2), \ref{lemXpeval} and the boundedness of the derivative of $\sigma$, we observe that
\begin{align*}
  \EXP{|U_t|^p 1_{\{t<\tau_m\}}} &\leq 2^{p-1} \EXP{\left|\int_0^t K(t-s) \nabla b^i(X_s)^\top U_s \mr{d}s\right|^p}\\
  &\hspace{1em} + C_{p,m,d}\sum_{j,k,l} \EXP{\left|\int_0^t K(t-s) \partial_k \sigma^i_j(X_s) \{U_s^k\mr{d}W^j_s + C_H\sigma^k_l(X_s)\mr{d}B^{l,j}_s\}\right|^p}\\
  &\leq C_1 + C_2 \int_0^t \sum_{j,k,l} \EXP{(|\nabla b^i(X_s)|^p + |\partial_k \sigma^i_j(X_s)|^p) |U_s|^p1_{\{t<\tau_m\}} +  |\partial_k \sigma^i_j(X_s)C_H\sigma^k_l(X_s)|^p|} \mr{d}s \\
  &\leq C_1 + C_2 \int_0^t \EXP{ |U_s|^p1_{\{t<\tau_m\}} } \mr{d}s.
\end{align*}
for some $C_1, C_2 >0$, independent of $t$. Therefore, Gronwall's lemma yields
\begin{equation*}
  \EXP{|U_t|^p 1_{\{t<\tau_m\}}} \leq C_1 e^{C_2t} \leq C_1 e^{C_2T} ,
\end{equation*}
and thus, by Fatou's lemma, we see
\begin{equation*}
  \EXP{|U_t|^p} = \EXP{\liminf_{m \to \infty} |U_t|^p 1_{\{t<\tau_m\}}} \leq \liminf_{m \to \infty} \EXP{|U_t|^p 1_{\{t<\tau_m\}}}  \leq C_1 e^{C_2T},
\end{equation*}
that is, $U_t$ is in $L_p$ for any $t\in[0,T]$.

Next we discuss the continuity. It is sufficient to verify the continuity of the stochastic integral part. Similarly to (\ref{equnhld}), we see for any $t, t+h\in[0,T], h>0$,
\begin{equation*}
  \EXP{|U_{t+h} - U_t|^p} \leq Ch^{Hp} \EXP{ |U_s|^p +\sum_{k,l}|\partial_k \sigma^i_j(X_s)C_H\sigma^k_l(X_s)|^p} +  \leq Ch^{Hp} .
\end{equation*}
Then, Kolmogorov's continuity theorem ensure the continuity of $U$.

Finally, we show the uniqueness in law of the solution. Assume there exist two continuous solutions $\nu^{1}, \nu^{2}$ satisfying (\ref{eqlimitSVE}). Then
\begin{equation*}
  \nu^{1}_t - \nu^{2}_t =  \int_0^t K(t-s) \sigma^\prime(X_s) (\nu^{1}_s -\nu^{2}_s) \mr{d}W_s
\end{equation*}
follows, and we will evaluate this difference in $L_p, p>2\beta^\ast$. Write $\tilde{\nu} = \nu^1 - \nu^2$.
By Lemma \ref{lemcalcforlem}-(1) and the boundedness of $\sigma^\prime$, we have
\begin{align*}
  \EXP{|\tilde{\nu}_t|^{p} } &\leq C \int_0^t \EXP{ |\sigma^\prime(X_s)|^p |\tilde{\nu}_s|^p } \mr{d}s \\
  &\leq C \int_0^t \EXP{|\tilde{\nu}_s|^p} \mr{d}s.
\end{align*}
Since $\tilde{\nu}$ is continuous, by Gronwall's lemma, we see
\begin{equation*}
  \EXP{|\tilde{\nu}_t|^2} = 0,
\end{equation*}
which implies that for all $t$, $\nu^{1}_t = \nu^{2}_t$ almost surely. Taking the continuities of $\nu^{1}$ and $\nu^{2}$ into account, we can verify these processes are indistinguishable. As a consequence, the uniqueness in law is derived from the Yamada-Watanabe argument.
\qed

\appendix

\section{Stochastic integrals as the fractional derivatives}
We introduce the representation of stochastic integral as the fractional
derivatives; see also Horvath et al.~\cite{Horvath}. We let $\lambda \in (0,1)$ and $K(t)=C_K t^{-\alpha}, \alpha \in (0,\lambda)$ in this section.

\begin{definition}
  Let $f \in \mcl{C}^\lambda_0$. We define the integral operator $\mcl{J}^\alpha$ as
  \begin{equation*}
    (\mcl{J}^\alpha f)(t) := K(t) f(t) - \int_0^t  K^\prime(t-s)(f(t) - f(s) ) \mr{d}s,
  \end{equation*}
  where $K^\prime = \mr{d} K/ \mr{d}t$.
\end{definition}

\begin{proposition}\label{propregfracderiv}
  Let $Y$ be a continuous semimartingale taking values in $\mcl{C}^\lambda_0$ such that the stochastic integral $\int_0^t K(t-s) \mr{d}Y_s$ is well defined for all $t\in[0,T]$. Then the integral is almost surely represented by $\mcl{J}^\alpha$ as
  \begin{equation*}
    \int_0^t  K(t-s) \mr{d}Y_s = (\mcl{J}^\alpha Y) (t) 
  \end{equation*}
for all $ t\in[0,T]$.
\end{proposition}

\begin{proof}
  Let $\epsilon>0$ be a sufficiently small number. It\^ o's integration by parts yields
  \begin{equation}\label{eqvkep1}
    \begin{split}
      \int_0^{t-\epsilon}  K^\prime(t-s)(Y_t - Y_s) \mr{d}s &= \left(\int_0^{t-\epsilon} K^\prime(t-s) \mr{d}s\right)Y_t  - \int_0^{t-\epsilon} K^\prime(t-s)Y_s  \mr{d}s \\
      &= \left( \int_{\epsilon}^{t} K^\prime(s) \mr{d}s\right)Y_t + [K(t-s)Y_s]_{s=0}^{s=t-\epsilon} - \int_0^{t-\epsilon} K(t-s) \mr{d}Y_s \\
      &= (K(t) -K(\epsilon)) Y_t  -  K(\epsilon)Y_{t-\epsilon} - \int_0^{t-\epsilon} K(t-s) \mr{d}Y_s \\
      &= K(t)Y_t - \int_0^{t-\epsilon} K(t-s) \mr{d}Y_s + K(\epsilon)(Y_{t-\epsilon}-Y_t) .
    \end{split}
  \end{equation}
  Since $V$ satisfies the \hd condition of an order $\lambda>\alpha$, the last term vanishes as $\epsilon\to 0$. Indeed, letting $A(\omega)$ be the \hd constant of $Y(\omega)$, we see
  \begin{equation*}
    |K(\epsilon)(Y_{t-\epsilon}(\omega)-Y_t(\omega))| \leq  \epsilon^{-\alpha} C_KA(\omega) \epsilon^{\lambda} =C_K A(\omega)\epsilon^{\lambda-\alpha},
  \end{equation*}
  which vanishes as $\epsilon$ tends to zero uniformly in $t$. The second term of on the right-hand side of (\ref{eqvkep1}) obviously tends to the original stochastic integral. On the other hand, the integral on the left-hand side of (\ref{eqvkep1}) converges to the integration on $[0,t]$ by the DCT since
  \begin{equation*}
    K^\prime(t-s) |Y_t(\omega) - Y_s(\omega)|  \leq C_K A(\omega) |t-s|^{\lambda-\alpha-1}
  \end{equation*}
  for all $s\in[0,t]$ and
  \begin{equation*}
    \int_0^t C_K A(\omega) |t-s|^{\lambda-\alpha-1} \mr{d}s \leq \frac{C_K}{\lambda-\alpha} A(\omega) t^{\lambda-\alpha}.
  \end{equation*}
  Therefore, taking the limit of both sides of (\ref{eqvkep1}) with $\epsilon \rightarrow 0$, we have
  \begin{equation}
    \int_0^{t} K(t-s) \mr{d}Y_s  = K(t) Y_t - \int_0^t K^\prime(t-s) (Y_t - Y_s ) \mr{d}s = (\mcl{J}^\alpha Y)(t),
  \end{equation}
  which is the desired result.
\end{proof}

The operator $\mcl{J}^\alpha$ is one of the fractional derivative
operators, called \textit{the Marchaud fractional derivative}, which coincides with the Riemann-Liouville fractional derivative on $\mcl{C}^{\lambda}_0$.

\begin{lemma}[Samko et al.\cite{RLSamko}] \label{lemJalphaconti}
  The operator $\mcl{J}^\alpha$ is bounded (continuous) from $\mcl{C}^\lambda_0$ into $\mcl{C}^{\lambda-\alpha}_0$.
\end{lemma}

\section{Tightness criterion for a space of \hd continuous processes}\label{sectight}
We will discuss the tightness of \hd continuous processes. We first introduce a compact set in a space of \hd continuous functions.

\begin{lemma}\label{lemhldcompact}
  A set of $\alpha$-\hd continuous functions
  \begin{equation*}
    K_\delta = \left\{f\in \mcl{C}^{\alpha}_0 \,\middle|\, \|f\|_{\mcl{C}^{\alpha}_0} \leq \delta\right\}
  \end{equation*}
  is compact in $\mcl{C}^{\beta}_0$ for any $\delta>0$ and $\beta\in(0,\alpha)$.
\end{lemma}

\begin{proof}
  We first prove $K_\delta$ is compact in $\mcl{C}_0$. Take any sequence $\{f_n\}_{n\in\Natu} \subset K_\delta$ such that $\|f_n-f\|_{\infty} \to 0$ for some $f$. We have for all $t,s \in [0,T], s\neq t$,
  \begin{align*}
    \frac{|f_n(t)-f_n(s)|}{|t-s|^\alpha} \leq \delta - \|f_n\|_\infty.
  \end{align*}
  Letting $n\to\infty$, we obtain
  \begin{align*}
    \frac{|f(t)-f(s)|}{|t-s|^\alpha}& \leq \delta - \|f\|_\infty.
  \end{align*}
  since $f_n(t)$ converges to $f(t)$ for all $t\in[0,T]$. Then, we obtain $\|f\|_{\mcl{C}^{\alpha}_0}\leq \delta$, and thus, $f\in K_\delta$. Hence, $K_\delta$ is closed in $\mcl{C}_0$, and along with the Arzel\' a-Ascoli theorem, we observe that $K_\delta$ is compact in $\mcl{C}_0$.

  We now prove the compactness of $K_\delta$ in $\mcl{C}^{\beta}_0$. By the compactness in $\mcl{C}_0$, for an arbitrary sequence $\{g_n\}_{n\in\Natu} \subset K_\delta$, there exist a subsequence $\{g_{n_k}\}_{k\in\Natu}$ and $g\in K_\delta$ such that $\|g_{n_k} - g\|_\infty \xrightarrow{k \to \infty} 0$.
  Since $g_{n_k}$ and $g$ belong to $K_\delta$,
  \begin{equation*}
    |g_{n_k}(t) - g(t) - (g_{n_k}(s) - g(s))| \leq 2\delta |t-s|^\alpha,\,^\forall k\in \Natu, \,^\forall t,s \in [0,T]
  \end{equation*}
  follows, and it also holds
  \begin{equation*}
    |g_{n_k}(t) - g(t) - (g_{n_k}(s) - g(s))| \leq |g_{n_k}(t) - g(t)| + |g_{n_k}(s) - g(s)| \leq 2 \|g_{n_k} - g\|_\infty, \,^\forall k\in \Natu, \,^\forall t,s \in [0,T].
  \end{equation*}
  Hence, the inequality
  \begin{equation*}
    a \wedge b \leq a^\theta b^{1-\theta}, \quad a,b>0, \theta\in[0,1]
  \end{equation*}
  leads to
  \begin{align*}
    |g_{n_k}(t) - g(t) - (g_{n_k}(s) - g(s))| &\leq (2\delta|t-s|^\alpha) \wedge (2 \|g_{n_k} - g\|_\infty) \\
    &\leq 2\delta^{\frac{\beta}{\alpha}} |t-s|^\beta \|g_{n_k} - g\|_\infty^{1-\frac{\beta}{\alpha}}
  \end{align*}
  for all $t,s$, where $0<\beta<\alpha$. Thus, we have
  \begin{align*}
    \|g_{n_k} - g\|_{\mcl{C}^{\beta}_0} &= \|g_{n_k} - g\|_\infty + \sup_{t,s\in [0,T]} \frac{ |g_{n_k}(t) - g(t) - (g_{n_k}(s) - g(s))|}{|t-s|^\beta} \\
    &= \|g_{n_k} - g\|_\infty + C \|g_{n_k} - g\|_\infty^{1-\frac{\beta}{\alpha}} \xrightarrow{k\to \infty} 0,
  \end{align*}
  which implies $K_\delta$ is compact in $\mcl{C}^{\beta}_0$.
\end{proof}

\begin{theorem}\label{thmvntight}
  Let $\{Y^{n}\}_{n\in\Natu}$ be a sequence of $\mcl{C}^{\alpha}_0$-valued random variables. If $\mathbb{E}[\|Y^{n}\|_{\mcl{C}^{\alpha}_0}]$ is bounded uniformly in $n$, the sequence $\{Y^{n}\}_{n\in\Natu}$ is tight in $\mcl{C}^{\beta}_0$ for $0<\beta<\alpha$.
\end{theorem}

\begin{proof}
  We are to show that for any $\epsilon>0$ there exists a compact set $K$ in $\mcl{C}^{\beta}_0$ such that $\PROB{Y^{n} \not \in K} < \epsilon$.
  Let $\epsilon>0$ and assume $K_\delta$ be the same set as in Lemma \ref{lemhldcompact}. Since $\mathbb{E}[\|Y^{n}\|_{\mcl{C}^{\alpha}_0}] < C$ ,
  taking $\delta=C\epsilon^{-1} + 1$ and Markov's inequality yield that
  \begin{equation*}
    \PROB{Y^{n} \not \in K_\delta} = \PROB{\|Y^{n}\|_{\mcl{C}^{\alpha}_0}>\delta } \leq \frac{\mathbb{E}[\|Y^{n}\|_{\mcl{C}^{\alpha}_0}]}{\delta} \leq \frac{C}{\delta} < \epsilon.
  \end{equation*}
  The result follows from the compactness of $K_\delta$ verified in Lemma \ref{lemhldcompact}.
\end{proof}

\begin{corollary}\label{corvnconvhld}
  Let $Y^{n}$ be a stochastic process which converges to a process $Y$ weakly in $\mcl{C}_0$ as $n$ goes to infinity. If $Y^{n}$ satisfies $\mathbb{E}[\|Y^{n}\|_{\mcl{C}^{\alpha}_0}] \leq C$ for some $C$ uniformly in $n$, it converges to $Y$ weakly in $\mcl{C}^{\beta}_0$ for any positive $\beta<\alpha$.
\end{corollary}

\begin{proof}
  Let $0<\beta<\alpha$ and let $\{Y^{n_k}\}_{k\in\Natu}$ be an arbitrary subsequence of $\{Y^{n}\}_{n\in\Natu}$. Then, by Theorem \ref{thmvntight}, together with Prokhorov's theorem, we see that there exist a subsequence $\{Y^{n^\prime_k}\}_{k\in\Natu}$ of $\{Y^{n_k}\}_{k\in\Natu}$ and $\tilde{Y} \in \mcl{C}^{\beta}_0$ such that $Y^{n^\prime_k}$ tends to $\tilde{Y}$ weakly in $\mcl{C}^{\beta}_0$ and, in particular, in $\mcl{C}_0$. However, by the assumption, $Y^{n^\prime_k}$ tends to $V$ weakly in $\mcl{C}_0$. Therefore, it must holds $\tilde{Y}=Y$. Since each subsequence of $\{Y^{n}\}_{n\in\Natu}$ includes a subsequence which converges to $Y$ weakly in $\mcl{C}^{\beta}_0$, we see that the original sequence converges to $Y$ weakly in $\mcl{C}^{\beta}_0$, too. This completes the proof.
\end{proof}

\section{Limit distribution of the integral of fractional parts}
The following lemma is analogous to that of Delattre and Jacod (Lemma 6.1 \cite{DelattreJacod}) and Tukey \cite{Tukey1938}.

\begin{lemma}
  Let $g\in L_1(0,1)$ be either nonnegative or nonpositive and let $\{Y_n\}_{n\in\Natu}$ be a sequence of random variables on $[0,t]$ whose density functions are each
  \begin{equation*}
    f_{Y_n}(s) = C_{n,t}g(ns - [ns]), \quad 0< s < t,
  \end{equation*}
  where
  \begin{equation*}
    C_{n,t} = \left(\frac{[nt]}{n}\int_0^1 g(r)\mr{d}r + \frac{1}{n}\int_0^{nt-[nt]}g(r)\mr{d}r\right)^{-1}
  \end{equation*}
  is the normalizing constant. Then $Y_n$ converges in law to the uniform distribution on $[0,t]$ as $n$ goes to infinity.
\end{lemma}

\begin{proof}
  Firstly, we confirm that $f_{Y_n}(y)$ is certainly a probability density function. It is easily checked by the simple calculation:
  \begin{equation}\label{eqgnsnscalc1}
    \begin{split}
      \int_0^t g(ns-[ns])\mr{d}s&=\sum_{j=0}^{[nt]-1} \int_{\frac{j}{n}}^{\frac{j+1}{n}}g(ns-j)\mr{d}s + \int_{\frac{[nt]}{n}}^{t}  g(ns-[nt])\mr{d}s\\
      &=\frac{1}{n}\left( \sum_{j=0}^{[nt]-1} \int_0^{1}  g(r)\mr{d}r + \int_0^{nt-[nt]}  g(r)\mr{d}r\right)\\
      &=\frac{[nt]}{n} \int_0^1 g(r)\mr{d}r +  \frac{1}{n}\int_0^{nt-[nt]}  g(r)\mr{d}r \\
      &=(C_{n,t})^{-1}.
    \end{split}
  \end{equation}
  \par We now show the convergence of the characteristic function of $Y_n$. For all $x \in \Real$ and $i = \sqrt{-1}$, we have
  \begin{align*}
    \int_0^t e^{ixs} f_{Y_n}(s) \mr{d} s &=\int_0^t e^{ixs} C_{n,t}g(ns-[ns]) \mr{d} s \\
    &= C_{n,t} \sum_{k=0}^{[nt]-1} \int_{\frac{k}{n}}^{\frac{k+1}{n}} e^{ixs} g(ns-k) \mr{d} s + C_{n,t} \int_{\frac{[nt]}{n}}^{t} e^{ixs} g(ns-[nt]) \mr{d} s\\
    &= \frac{C_{n,t}}{n}\sum_{k=0}^{[nt]-1} \int_0^1 e^{ix\frac{r+k}{n}} g(r) \mr{d} r + \frac{C_{n,t}}{n} \int_0^{nt-[nt]} e^{ix\frac{r+[nt]}{n}} g(r) \mr{d} r.
  \end{align*}
  Since $0< nt-[nt] < 1$,
  \begin{equation*}
    \left|\frac{1}{n} \int_0^{nt-[nt]} g(r) \mr{d} r\right| \leq  \frac{1}{n} \int_{0}^{1} |g(r)| \mr{d} r \rightarrow 0,
  \end{equation*}
  and therefore, by the convergence
  \begin{equation*}
    C_{n,t} = \left( \frac{[nt]}{n}\int_0^1 g(r)\mr{d}r + \frac{1}{n} \int_0^{nt-[nt]} g(r)\mr{d}r \right)^{-1} \xrightarrow{n\to\infty} \left(t\int_0^1 g(r)\mr{d}r\right)^{-1},
  \end{equation*}
  we see
  \begin{equation*}
    \left|\frac{C_{n,t}}{n} \int_0^{nt-[nt]} e^{ix\frac{r+[nt]}{n}} g(r) \mr{d} r\right| \leq  \frac{C_{n,t}}{n} \int_{0}^{1} |g(r)| \mr{d} r \rightarrow 0.
  \end{equation*}
  By the triangle inequality, it holds
  \begin{equation*}
    \left|\frac{1}{n}\sum_{k=0}^{[nt]-1} e^{ix\frac{k}{n}} - \int_0^t e^{ixs} \mr{d} s \right|\leq \left|\frac{1}{n}\sum_{k=0}^{[nt]-1} e^{ix\frac{k}{n}} - \int_0^{\frac{[nt]}{n}} e^{ixs} \mr{d} s \right| + \int_{\frac{[nt]}{n}}^t e^{ixs} \mr{d}s.
  \end{equation*}
  The last term on the right-hand side vanishes as $n$ goes to infinity. The other term on the right-hand side is evaluated as
  \begin{align*}
    \left|\frac{1}{n}\sum_{k=0}^{[nt]-1} e^{ix\frac{k}{n}} - \int_0^{\frac{[nt]}{n}} e^{ixs} \mr{d} s \right| &=\left| \frac{1}{n}\sum_{k=0}^{[nt]-1} \left(e^{ix\frac{k}{n}}-n\int_{\frac{k}{n}}^{\frac{k+1}{n}} e^{ixs} \mr{d} s \right)\right|\\
    &=\left| \frac{1}{n}\sum_{k=0}^{[nt]-1} \left(e^{ix\frac{k}{n}}-\int_{0}^{1} e^{ix\frac{s+k}{n}} \mr{d} s \right)\right|\\
    &\leq \left|1-\int_{0}^{1} e^{ix\frac{s}{n}} \mr{d} s \right| \cdot \frac{1}{n}\sum_{k=0}^{[nt]-1} \left| e^{ix\frac{k}{n}} \right|\\
    &\leq \int_{0}^{1} \left|1- e^{ix\frac{s}{n}} \right| \mr{d} s \frac{[nt]}{n}\rightarrow 0.
  \end{align*}
  Hence, we have
  \begin{equation*}
    \lim_n\frac{1}{n}\sum_{k=0}^{[nt]-1} e^{ix\frac{k}{n}} = \int_0^t e^{ixs} \mr{d} s,
  \end{equation*}
  and then, along with the DCT, it is implied that
  \begin{align*}
    C_{n,t} \frac{1}{n}\sum_{k=0}^{[nt]-1} \int_0^1 e^{ix\frac{r+k}{n}} g(r) \mr{d} r &= C_{n,t}  \left(\frac{1}{n}\sum_{k=0}^{[nt]-1} e^{ix\frac{k}{n}}\right) \int_0^1 e^{ix\frac{r}{n}} g(r) \mr{d} r\\
    &= \left(t\int_0^1 g(r)\mr{d}r\right)^{-1} \int_0^{t} e^{ixs} \mr{d} s \int_0^1 g(r) \mr{d} r\\
    &= \int_0^{t} \frac{1}{t}e^{ixs} \mr{d} s.
  \end{align*}
  Indeed, a dominating function is derived as
  \begin{equation*}
    \left|\left(\frac{1}{n}\sum_{k=0}^{[nt]-1} e^{ix\frac{k}{n}}\right) e^{ix\frac{r}{n}} g(r) \right| \leq \frac{[nt]}{n}|g(r)| \leq T |g(r)|.
  \end{equation*}
  Thus,
  \begin{equation*}
    \int_0^t e^{ixs} f_{Y_n}(s) \mr{d} s \rightarrow \int_0^t \frac{1}{t} e^{ixs}\mr{d}s
  \end{equation*}
  holds. Since the function $s \mapsto 1/t$ is the density function of
 the uniform distribution on $[0,t]$, this means the convergence of
 the characteristic functions, which concludes the proof.
\end{proof}

By this lemma, for any continuous function $k$, we have
\begin{equation}\label{equnidisconvas}
  \begin{split}
    &\int_0^t k(s) g(ns-[ns]) \mr{d} s= (C_{n,t})^{-1} \int_0^t k(s) f_{Y_n}(s)\mr{d} s \\
    &\xrightarrow{n\to \infty}t\int_0^1 g(r)\mr{d}r\int_0^t k(s) \frac{\mr{d}s}{t} = \int_0^1 g(r)\mr{d}r \int_0^t k(s) \mr{d} s
  \end{split}
\end{equation}
by the property of convergence in law. We will apply this result to stochastic processes.

\begin{lemma}\label{lemrefineddelattre}
  Assume further $g\in L_2(0,1)$. Let $H^{(n)}$ and $H$ be stochastic processes on $[0,T]$ such that
  \begin{equation*}
    \EXP{\int_0^T |H^{(n)}_s-H_s|^2 \mr{d} s} \rightarrow 0
  \end{equation*}
  with $H$ being almost surely continuous. Then, for all $t\in [0,T]$,
  \begin{equation*}
    \int_0^t H^{(n)}_s g(ns-[ns])\mr{d}s \xrightarrow[\text{in }L_2]{n\to\infty}\int_0^1g(r)\mr{d}r \int_0^t H_s \mr{d} s.
  \end{equation*}
\end{lemma}

\begin{proof}
  The following evaluation is derived in a similar way to $(\ref{eqgnsnscalc1})$ and we use it several times in this proof:
  \begin{equation}\label{eqgnsnscalc2}
    \int_0^t |g(ns-[ns])|^2 \mr{d}s = \frac{[nt]}{n} \int_0^1 |g(r)|^2 \mr{d}r + \int_0^{nt-[nt]} |g(r)|^2\mr{d}r \leq C,
  \end{equation}
  where $C$ does not depend on $n$. It follows from Minkowski's inequality,
  \begin{align*}
    &\hspace{1em}\left\|\int_0^t H^{(n)}_s g(ns-[ns]) \mr{d}s - \int_0^1 g(r)\mr{d}r\int_0^t H_s \mr{d} s\right\|_{L_2} \\
    &\leq \left\|\int_0^t (H^{(n)}_s-H_s) g(ns-[ns]) \mr{d}s\right\|_{L_2} + \left\|\int_0^t H_s \left(g(ns-[ns])-\int_0^1 g(r)\mr{d}r\right)\mr{d} s\right\|_{L_2}.
  \end{align*}
  By the Cauchy-Schwarz inequality and (\ref{eqgnsnscalc2}), the first term on the right-hand side satisfies
  \begin{align*}
    \EXP{\left|\int_0^t (H^{(n)}_s-H_s) g(ns-[ns])\mr{d}s\right|^2} &\leq \left(\int_0^t  |g(ns-[ns])|^2 \mr{d}s \right) \EXP{\int_0^t |H^{(n)}_s-H_s|^2\mr{d}s}\\
    &\leq C\EXP{\int_0^t |H^{(n)}_s-H_s|^2\mr{d}s} \rightarrow 0.
  \end{align*}
  Since $H$ is continuous, according to (\ref{equnidisconvas}),
  \begin{equation*}
    \left|\int_0^t H_s \left(g(ns-[ns])-\int_0^1 g(r)\mr{d}r\right)\mr{d} s\right|^2 \rightarrow 0
  \end{equation*}
  holds almost surely. We have also
  \begin{equation*}
    \left|\int_0^t H_s \left(g(ns-[ns])-\int_0^1 g(r)\mr{d}r\right)\mr{d} s\right|^2 \leq \int_0^t H_s^2 \mr{d}s \int_0^t \left(g(ns-[ns])-\int_0^1 g(r)\mr{d}r\right)^2\mr{d}s
  \end{equation*}
  by the Cauchy-Schwarz inequality. The hypothesis $\EXP{\int_0^t H_s^2 \mr{d}s}<\infty$ and the evaluation
  \begin{equation*}
    \int_0^t \left(g(ns-[ns])-\int_0^1 g(r)\mr{d}r\right)^2\mr{d}s \leq \int_0^t |g(ns-[ns])|^2 \mr{d}r +2T \left(\int_0^1 g(r)\mr{d}r\right)^2< \infty,
  \end{equation*}
  which is derived from (\ref{eqgnsnscalc2}), enable us to apply the DCT with respect to the integration of $\mathsf{P}$ to obtain
  \begin{equation*}
    \EXP{\left|\int_0^t H_s \left(g(ns-[ns])-\int_0^1 g(r) \mr{d}r\right)\mr{d} s\right|^2} \rightarrow 0,
  \end{equation*}
  which concludes the proof.
\end{proof}

\noindent
{\bf Acknowledgement:}
The authors are grateful to Yushi Hamaguchi for helpful discussions.

\bibliography{./reference}

\begin{thebibliography}{10}

\bibitem{AVP}
E.~Abi~Jaber, M.~Larsson, and S.~Pulido.
\newblock Affine {Volterra} processes.
\newblock {\em The Annals of Applied Probability}, 29(5):3155 -- 3200, 2019.

\bibitem{Aida-Naganuma}
S.~Aida and N.~Naganuma.
\newblock Error analysis for approximations to one-dimensional {SDE}s via the
  perturbation method.
\newblock {\em Osaka Journal of Mathematics}, 57:381--424, 2020.

\bibitem{Kebaier2015}
M.~B. Alaya and A.~Kebaier.
\newblock Central limit theorem for the multilevel {Monte Carlo Euler} method.
\newblock {\em The Annals of Applied Probability}, 25(1):211--234, 2015.

\bibitem{Billingsley}
P.~Billingsley.
\newblock {\em Convergence of probability measures}.
\newblock Wiley, 1999.

\bibitem{DelattreJacod}
S.~Delattre and J.~Jacod.
\newblock A central limit theorem for normalized functions of the increments of
  a diffusion process, in the presence of round-off errors.
\newblock {\em Bernoulli}, pages 1--28, 1997.

\bibitem{Fukasawa-Hirano}
M.~Fukasawa and A.~Hirano.
\newblock Refinement by reducing and reusing random numbers of the {H}ybrid
  scheme for {B}rownian semistationary processes.
\newblock {\em Quantitative Finance}, 21(7):1127--1146, 2021.

\bibitem{Fukasawa2020}
M.~Fukasawa and J.~Ob{\l}{\'o}j.
\newblock Efficient discretisation of stochastic differential equations.
\newblock {\em Stochastics}, 92(6):833--851, 2020.

\bibitem{Haeusler}
E.~Haeusler and H.~Luschgy.
\newblock {\em Stable Convergence and Stable Limit Theorems (Probability Theory
  and Stochastic Modelling, 74)}.
\newblock Springer, 2015.

\bibitem{Hamadouche2000}
D.~Hamadouche.
\newblock Invariance principles in {H\"older} spaces.
\newblock {\em Portugaliae Mathematica}, 57(2):127--151, 2000.

\bibitem{Horvath}
B.~Horvath, A.~Jacquier, and A.~Muguruza.
\newblock Functional central limit theorems for rough volatility.
\newblock {\em arXiv preprint arXiv:1711.03078}, 2017.

\bibitem{Hu-Liu-Nualart}
Y.~Hu, Y.~Liu, and D.~Nualart.
\newblock Rate of convergence and asymptotic error distribution of {E}uler
  approximation schemes for fractional diffusions.
\newblock {\em The Annals of Applied Probability}, 26(2):1147--1207, 2016.

\bibitem{Jacod1997}
J.~Jacod.
\newblock On continuous conditional {Gaussian} martingales and stable
  convergence in law.
\newblock In {\em S\`{e}minaire de Probabilit\`{e}s XXXI}, pages 232--246.
  Springer, 1997.

\bibitem{JacodProtterAsymp}
J.~Jacod and P.~E. Protter.
\newblock Asymptotic error distributions for the {Euler} method for stochastic
  differential equations.
\newblock {\em The Annals of Probability}, 26(1):267 -- 307, 1998.

\bibitem{Jacod2011}
J.~Jacod and P.~E. Protter.
\newblock {\em Discretization of processes}.
\newblock Springer, 2011.

\bibitem{KurtzProtterWong}
T.~G. Kurtz and P.~E. Protter.
\newblock {Wong}-{Zakai} corrections, random evolutions, and simulation schemes
  for {SDE's}.
\newblock In E.~Mayer-Wolf, E.~Merzbach, and A.~Shwartz, editors, {\em
  Stochastic analysis : liber amicorum for Moshe Zakai}, pages 331--346.
  Academic Press, 1991.

\bibitem{Kurtz1996}
T.~G. Kurtz and P.~E. Protter.
\newblock Weak convergence of stochastic integrals and differential equations.
\newblock In D.~Talay and L.~Tubaro, editors, {\em Probabilistic Models for
  Nonlinear Partial Differential Equations}, pages 1--41. Springer, 1996.

\bibitem{Liu-Tindel}
Y.~Liu and S.~Tindel.
\newblock First-order {E}uler scheme for {SDE}s driven by fractional {B}rownian
  motions: The rough case.
\newblock {\em The Annals of Applied Probability}, 29(2):758--829, 2019.

\bibitem{Mishura}
Y.~Mishura.
\newblock {\em Stochastic Calculus for Fractional Brownian Motion and Related
  Processes}.
\newblock Springer, 2008.

\bibitem{RevuzYor}
D.~Revuz and M.~Yor.
\newblock {\em Continuous Martingales and {Brownian} Motion}.
\newblock Number 293 in Grundlehren der mathematischen Wissenschaften.
  Springer, 1991.

\bibitem{Richard-Tan-Yang2021}
A.~Richard, X.~Tan, and F.~Yang.
\newblock Discrete-time simulation of stochastic {Volterra} equations.
\newblock {\em Stochastic Processes and their Applications}, 141:109--138,
  2021.

\bibitem{RLSamko}
S.~G. Samko, A.~A. Kilbas, and O.~I. Marichev.
\newblock {\em Fractional Integrals and Derivatives: Theory and Applications}.
\newblock Gordon and Breach Science Publishers, 1993.

\bibitem{Tukey1938}
J.~W. Tukey.
\newblock On the distribution of the fractional part of a statistical variable.
\newblock {\em Rec. Math. [Mat. Sbornik] N.S.}, 4(3):561--562, 1938.

\end{thebibliography}
\bibliographystyle{plain}

\end{document}